\documentclass[12pt]{amsart}

\usepackage{amsmath,amssymb,amsthm,xcolor}
\usepackage[margin=1.2in,marginparwidth=1in]{geometry}
\usepackage{comment}
\usepackage{hyperref}
\usepackage{mathtools}
\usepackage{xspace}

\usepackage{standard}

\newcommand{\tgt}{\top}
\newcommand{\nrml}{\perp}

\newcommand{\jump}{\triangle}
\newcommand{\ol}[1]{\overline{#1}}
\newcommand{\JJ}{\mathcal{J}}
\DeclareMathOperator{\ind}{ind}
\DeclareMathOperator{\second}{II}
\DeclareMathOperator{\sgn}{sgn}
\newcommand{\RN}[1]{%
  \textup{\uppercase\expandafter{\romannumeral#1}}%
}
\newcommand{\cov}{\mathnormal{D_t}}

\newcommand{\per}{\text{per}}
\DeclareMathOperator{\shape}{\mathsf{S}}
\DeclareMathOperator{\Hess}{\mathsf{Hess}}

\newcommand{\phys}{physical path\xspace}
\newcommand{\physs}{physical paths\xspace}
\newcommand{\reflective}{reflected\xspace}
\newcommand{\Ct}{\mathcal{C}^\infty}

\newcommand{\Rt}{\mathcal{R}}
\newcommand{\Kt}{\mathcal{K}}

\title[Morse index]{The Morse index theorem for mechanical systems
  with reflections}

\author{Jared Wunsch}
\email{jwunsch@math.northwestern.edu}
\address{Department of Mathematics, Northwestern University, Evanston, IL, USA}

\author{Mengxuan Yang}
\email{mxyang@math.berkeley.edu}
\address{Department of Mathematics, University of California, Berkeley, Berkeley, CA, USA}

\author{Yuzhou Zou}
\email{yuzhou.zou@northwestern.edu}
\address{Department of Mathematics, Northwestern University, Evanston, IL, USA}

\begin{document}
\maketitle
\begin{abstract}
We prove a Morse index theorem for action functionals on paths that are allowed to reflect
at a hypersurface (either in the interior or at the boundary of a
manifold).  Both fixed and periodic boundary conditions are treated.
    \end{abstract}

\section{Introduction}

The classical Morse index theorem \cite{Mo:32} on a smooth Riemannian manifold
$(M,g)$ says that at a geodesic $\alpha(t)$, $t
\in [0,T]$, the index of the second variation of the energy
functional with fixed endpoints equals the total number (with multiplicity) of points on
$\alpha$ conjugate to $\alpha(0)$.  Indeed, Morse's celebrated book \cite{Mo:32} treats a
number of variations on this theme, allowing for different boundary
conditions, such as periodicity, which entail interesting corrections
to the count of conjugate points; in the periodic case, Morse refers
to the additional term as the ``order of concavity.''

Here we are concerned with a generalization of these classical
results, where we also allow \emph{reflections}.  We
simultaneously treat two cases: either $M$ is a manifold with boundary $Y=\pa
M$, or $Y$ is an embedded interior hypersurface of $M$.  
The paths
under consideration are required to undergo reflection at $Y$ in the
case $Y=\pa M$, or permitted to undergo \emph{either} reflection or
transmission in the case of an interior hypersurface.

We are moreover concerned here not with the usual setting of
Riemannian geometry most common in the literature, but rather with the
more general case of a mechanical system: rather than just using an
energy functional given by the Riemannian metric,  we employ a
Lagrangian $L \in \CI (TM)$
$$
L(x,v) = \frac 12 g_{ij}(x) v^i v^j -V(x),
$$
and associated action
$$
J[\alpha] = \int_0^T L(\alpha(t), \dot \alpha(t)) \, dt.
$$
Here $V(x)$ is a real valued function that is not required to
be globally smooth: our hypotheses are that $V$ is smooth up to $Y$
but in the interior hypersurface case is merely required to have
matched values and first derivatives across the two sides.  (The
choice to work in this generality stems from our intended applications, described below.)

\emph{The main results of this paper are a Morse index theorem for the functional $J$,
both for trajectories with fixed boundary points and for the problem
of periodic trajectories}; these are stated in
Theorems~\ref{thm:morse1} and \ref{thm:morse} below.  We
  also prove an \emph{addition formula}
  (Theorem~\ref{thm:addition} below) that computes the difference of the Morse
  index (with fixed endpoint conditions) of a concatenation of two paths 
  and the sum of the Morse indices of the individual paths; the result is
  expressed as the index of the Hessian of the sum of the action
  functions along the two paths.

Our interest in this problem was stimulated by the aim of proving a
\emph{Gutzwiller trace formula} that would relate the asymptotics of
the trace of the Schr\"odinger propagator $e^{-it (\frac 12 \Lap+V)/h}$ to the
behavior of closed classical trajectories \cite{Gu:71}. Typically in such trace formulae, one
obtains a \emph{Maslov factor}, a power of $i$ that is the
Morse index of the variational problem for closed trajectories. For instance, in the
Riemannian geometric case related to the Duistermaat--Guillemin trace
formula \cite{Duistermaat-Guillemin1}, this variational characterization of
  the Maslov factor was established in
\cite{Du:76}.
In mechanical systems with a non-smooth potential $V$, singular across
$Y$ as described above, it turns out that in addition to
  ordinary periodic physical trajectories, there are contributions to the trace
  asymptotics from periodic trajectories that are \emph{reflected} at
  $Y$ as well as those transmitted
across $Y$ \cite{GaWu:18}. 
We were dismayed that we could not find any existing account of
the Morse index theorem for the periodic
variational problem with reflected mechanical trajectories; since this problem seems
a physically natural and important one, we have attempted to fill this
gap in the literature here.

A simple invocation of the usual proof of the index theorem
\emph{mutatis mutandis} does not suffice to deal with the case of
reflected paths.  To begin with, the spaces of allowable paths and
variations must be rather carefully set up: we must enforce
compatibility conditions at the reflection times (as well as at times
of transmission across $Y$), and this of course affects the space of
allowable variation vector fields.  The Jacobi fields, in turn, must
satisfy interesting geometric compatibility conditions at the moment
of reflection (involving the second fundamental form of the
hypersurface), and much of our work here has been to tell the story of
reflected Jacobi fields; the final proof of the index theorem is
straightforward once the tools to deal with the Jacobi fields are in place.

Some results in this direction do appear in the existing literature,
but not in the generality that we seek here.  In particular,
there are a number of treatments describing Jacobi fields for reflected geodesics by considering the
first variation of a family of broken trajectories reflecting at the
boundary according to Snell's law \cite{KaZh:03,Wo:07,ZhLi:07,IlSa:16,Ko:18}.
However, we have not been able to find analysis of the second
variation, nor a proof of the index theorem (neither for periodic nor fixed boundary
conditions).  It has also proved
impossible to find an account of the reflection conditions in the
presence of a potential.  Additionally, Morse's ``order of concavity'' arising in
the periodic variational problem moreover makes a somewhat obscure
appearance in \cite{Mo:32}, and is not easily suited to physical
interpretation in the mechanical context; the version here is not one
we have seen in the literature.  The addition formula
  of Theorem~\ref{thm:addition} is an essential ingredient in analyzing
  the Maslov factor in stationary phase expansions when composing the Schr\"odinger
  propagator with itself.  We have been unable to find a version of
  this in the existing literature even in the classical setting with
  smooth coefficients; we thus include it here as a classical dynamics theorem of
  potentially independent interest.

This paper is thus intended to provide a
thorough and, we hope, readable account of the generalization of the
classic theory of Jacobi fields and Morse indices to the general setting of mechanical systems with
reflections.

\subsection*{Structure of the paper}
We first work with paths with reflections off the hypersurface/boundary. In Section \ref{sec:path}, we consider permissible path spaces and variation vector fields in our
variations. In Section \ref{sec:1var}, we define the action of a path
and consider the first variation of the action. In Section
\ref{sec:2var}, we consider the second variation of the action and the
corresponding Jacobi fields. In Section \ref{sec:index}, we prove the Morse index theorem for both fixed endpoint paths and closed \reflective paths,
along with the Morse index addition formula.

\subsection*{Acknowledgments} JW received partial support from NSF grant DMS--2054424.

\section{Path space of \reflective trajectories}
\label{sec:path}

Consider a smooth Riemannian manifold with boundary $(M,g)$ and a compact embedded hypersurface $Y$, possibly disconnected.   
We assume that $\pa M \subset Y$. 
Near any point on $Y$, we may choose local Fermi coordinates $(x^1,\dots, x^n)$ so that locally $Y=\{x^1=0\}$;
thus $x^1$ is the signed distance to $Y$ (or, near
points where $Y$ coincides with $\pa M$, simply the distance) and locally the metric is of
the form
$$
g=(dx^1)^2+ \sum_{i,j=2}^n h_{ij}(x) dx^i dx^j
$$
For $W \in T_YM$ we adopt the notation $$W=W_\perp+W_\tgt$$ to
  denote the splitting of $W$ into normal and tangent components,
  using the metric, i.e., if $W = \sum W^j \pa_{x^j}$  in Fermi
  coordinates, then
  $$
W_\perp = W^1 \pa_{x^1},\quad W_\tgt = \sum_{j=2}^n W^j \pa_{x^j}
  $$
We will occasionally use the subscript $1$ to denote the $x^1$
component of curves or vector fields; in particular, then, with a choice of an oriented unit normal to $Y$, e.g.\ $N=\pa_{x^1}$
  in Fermi normal coordinates, we write
  $$
W_\perp=W_1 N.
$$
Throughout this paper we assume that
$V\in \mathcal{C}^{\infty}(M\backslash Y; \RR)\cap \mathcal{C}^1(M)$
and that $V$ is $\mathcal{C}^\infty$ smooth up to $Y$, separately from
each side if $Y$ is locally an interior hypersurface.  As noted above,
the allowed discontinuities of second or higher derivatives of $V$
across $Y$ are not especially interesting in the context of the geometric considerations
here, and are included for the sake of future applications to
Schr\"odinger operators (for which the derivative discontinuities of
$V$ reflect energy).

We will consider the variation problem for the action associated
  to the Lagrangian $L \in \CI(TM)$ given by
  $$
L(x,v) = \frac 12 g_{ij}(x) v^i v^j -V(x).
$$
Implicitly, then, we are dealing with the Hamiltonian dynamics for the
Hamiltonian function on $T^*M$ given by Legendre transform:
$$
\frac 12 g^{ij}(x) \xi_i \xi_j +V(x).
$$

\subsection{Path spaces and variations}

Fix a time $T$; this will be left implicit in our notation for path
spaces.  In what follows we will use the notation $\bullet(t\pm)$ for
$\lim_{\ep\downarrow 0} {\bullet(t\pm\ep )}$, with $\bullet$ denoting a
function, vector field, etc., depending on $t$.

Our path space is defined as follows:
\begin{definition}
    \label{def:path-n}
    Let $0=T_0<T_1<\dots <T_m<T_{m+1}=T$, with $\{T_i\,:\,1\le i\le m\}= \Rt \cup \Kt$ a partition into a set of \emph{reflection} times $\Rt$ and a set of \emph{kink} times $\Kt$.
    A \emph{\reflective path} with reflection times $\Rt$ and kink times $\Kt$ is a continuous %piecewise $\Ct$-
		map $\alpha: [0,T]\to M$ such that
                \begin{enumerate}
\item If $\alpha(t) \in \pa M$ then there exists $i$ such that $t=T_i \in \Rt$.
                \item For each $0\le i\le m$, $\alpha$ restricted to $[T_i,T_{i+1}]$ is smooth (i.e., smooth in the interior with derivatives extending to the boundary of each subinterval).
    \item If $T_i\in \Rt$, then $\alpha(T_i)\in Y$, and if $x^1$ is a defining function for $Y$ and $I_i\ni T_i$
    is a sufficiently small open interval then the sign of
  $x^1\circ \alpha$ is constant on $I_i \backslash \{T_i\}$.
  \item If $T_i\in \Rt$, then ${(x^1\circ\alpha)'(T_i\pm)}\ne 0$, i.e., $\alpha$ is not tangent to $Y$ at $T_i$ from either direction.
  \end{enumerate}
  More generally, a piecewise smooth path $\alpha$ equipped with
    a set $\Rt$ of reflection times is said to be
    a \reflective path if there exist some choice of kink times $\Kt$ such
    that the above definition applies.
\end{definition}
\begin{remark} 
\label{rem:pathrem}
\phantom{}

\begin{itemize}\item We allow $\alpha(t)\in Y$ even if $t \notin \Rt$; the
importance of the reflection times arises in the requirement that
the paths do intersect $Y$ transversely at the reflection times
and stay on the same side of $Y$ before and after these times,
and in 
the following
definition of allowed variations, which will ensure that physical paths must
be reflected rather than allowing transmission across $Y$
at times in $\Rt$.  The specification of the reflection times
  is part of the data of the path; the kink times, by contrast, are
  not.

\item
The kink times $\Kt$ should be thought of times, outside of
$\Rt$, where $\alpha$ is allowed to fail to be smooth; such a set of
times is in general not unique. Indeed, if $\alpha$ is a \reflective
path with reflection and kink times $\Rt$ and $\Kt$, and $\Kt'$ is any
finite set with $\Kt\subset\Kt'\subset[0,T]\backslash\Rt$, then
$\alpha$ is a \reflective path with reflection and kink times $\Rt$
and $\Kt'$ as well. 
Note that $\alpha$ always admits a \emph{minimal} kink time set, namely
\[\Kt_{min} = \{t\in(0,T)\backslash\Rt\,:\,\alpha\text{ is not }\Ct\text{ at }t\},\]
and any other set of kink times $\Kt$ satisfies $\Kt\supset\Kt_{min}$.
As we see below, we will sometimes introduce
additional kink times, where $\alpha$ is in fact smooth, in order to
consider variations which develop kinks at those times.
\end{itemize}
\end{remark}

Let $\alpha$
be a \reflective path from $\alpha(0)=p$ to $\alpha(T)=p'$, with reflection times $\Rt$ and
kink times $\Kt$ as in the definition above, so that $\Rt\cup\Kt = \{T_1,\dots,T_m\}$.

\begin{definition}
    \label{def:variation}
    A variation of $\alpha(t)$
		is a map
    $\alpha(t,\ep)$ from $[0, T] \times (-\ep_0,\ep_0)$ to $M$,
    together with a family of smooth functions $0<T_1(\ep)<\dots < T_m(\ep)<T$, divided into a family of reflected time functions $\tilde\Rt$ and kink time functions $\tilde\Kt$ with $T_i(\ep)\in\tilde\Rt\iff T_i\in\Rt$,
    such that 
\begin{enumerate}
    \item $\alpha(t,0)=\alpha(t)$ for $t\in [0,T]$ and $T_i(0)=T_i$
      for $i=1,\dots, m$.
		\item For any fixed $\ep$, $\alpha(\cdot,\ep)$ is a
                  \reflective path as defined in Definition
                  \ref{def:path-n}, with reflection times
                  $\{T_i(\ep)\,:\,T_i\in\Rt\}$ and kink times
                  $\{T_i(\ep)\,:\,T_i\in\Kt\}$. 
		\item $\alpha(t,\ep)$ is smooth, up to the boundary, on each set of the form
		\[\{(t,\ep)\,:\,T_i(\ep)\le t\le T_{i+1}(\ep),\ep\in(-\ep_0,\ep_0)\},\quad 0\le i\le m\]
		(where we interpret $T_0(\ep) = 0$ and $T_{m+1}(\ep) = T$).
\end{enumerate}
    A two-parameter variation is defined analogously, with $\ep \in (-\ep_0, \ep_0)$ replaced by $(\delta, \ep) \in  (-\delta_0, \delta_0) \times (-\ep_0, \ep_0).$
\end{definition}
\begin{remark} Note that we allow the kink times to vary in $\ep$, which is not the standard
prescription e.g.\ in \cite{Mi:63}. This is useful owing to the
non-smoothness of physical paths at times of transmission across $Y$,
which may vary in families. \end{remark}

Throughout this paper, we shall consider the following three families of paths (with possibly multiple reflections) 
\begin{enumerate}
\item the \emph{space of \reflective paths} $\Omega(M)$: see Definition \ref{def:path-n}.
\item the \emph{space of \reflective paths with fixed endpoints} $p,p'$: 
\[\Omega_0(M;p,p') = \{\alpha\in\Omega(M)\,:\,\alpha(0) = p,\,\alpha(T)=p'\}\]
\item the \emph{space of periodic \reflective paths} where the endpoints are not fixed but need to be equal, i.e.,
\[\Omega_{\per}(M) = \{\alpha\in\Omega(M)\,:\,\alpha(0)=\alpha(T)\}.\]
Note that we do not require the derivative to match for $t=0,T$, but
this is consistent with our convention that elements of
$\Omega(M)$ are only piecewise smooth; the endpoint $\alpha(0)$ thus
plays no distinguished role, as there may or may not be a derivative
discontinuity there.
\end{enumerate}
Then we have
\[\Omega_0(M;p,p)\subset \Omega_{\per}(M)\subset \Omega(M)\ \text{ and } \ \Omega_0(M;p,p')\subset\Omega(M).\]

For notational convenience, we introduce notation for jumps and
  averages of vector fields along $\alpha$ at reflections and kinks.  Recall that for $Z$ any vector field along $\alpha$ and $t \in
  [0, T]$, we let
	\[Z(t\pm) := \lim_{\ep\to 0^+}Z(t\pm\ep);\]
	sometimes we denote this $Z^{\pm}$ if the time of evaluation $t$ is understood. We additionally set
\[\jump Z(t)= Z(t+) - Z(t-) \text{ and } \overline{Z}(t)= \tfrac{1}{2}(Z(t+)+Z(t-)).\]

We now anticipate the outcome of our analysis of first variations by defining reflected \physs; the relevance of this definition is demonstrated by Lemma~\ref{lemma:reflective_path}. In the following definition (and henceforth) we denote the covariant derivative along a path by
$$
D_t := \nabla_{\dot \alpha}.
$$ 
\begin{definition}
    \label{def:path.global}
We say that a \reflective path $\alpha(t)$, $t \in I=[0,T]$  is a \emph{reflected \phys} if the following conditions hold:
\begin{enumerate}
\item $\alpha\in\Omega(M)$
\item\label{geodesicitem} $D_t\dot \alpha + \nabla V(\alpha(t)) = 0$ on $I \backslash \{T_1, \dots, T_m \}$.
\item \label{reflectionitem} For all $T_j\in\Rt$, $\ol{\dot \alpha_\nrml(T_j)}  = 0$ and $\jump \dot \alpha_\tgt(T_j)=0$.
\item \label{kinkitem} For all $T_j\in\Kt$, $\jump\dot\alpha(T_j) = 0$.
\end{enumerate}
In addition, if $\alpha(t)$ also satisfies endpoint conditions $\alpha(0)=\alpha(T)$ and $\dot\alpha(0)=\dot \alpha(T)$, it is called a \emph{periodic \reflective \phys}.
\end{definition}

\begin{remark} 
Since $\alpha$ and $\dot \alpha$ are continuous at points in $\Kt$ and
since $\alpha$ solves a second order ODE with smooth coefficients away
from $Y$, a \reflective \phys $\alpha$ is in fact smooth at \emph{interior} kinks.  Since
$\nabla^2 V$ is allowed to be discontinuous at $Y$, however, third
derivatives of $\alpha$ may
be discontinuous at kinks in $Y$, i.e., at times $T_j \in \Kt$ with $\alpha(T_j) \in Y$. Nonetheless, $\alpha$ will be $C^2$ at such kinks, and hence for a \reflective \phys $\alpha$, the reflective times $\Rt$ can be characterized as the times where $\alpha$ is continuous, but not $C^2$.
  \end{remark}

We now show that, given a \reflective \phys, we can perturb the initial position and velocity to uniquely produce another \reflective \phys which reflects at similar times. 
\begin{lemma}
\label{lem:existuniq}
Let $\alpha(t)$ be a \reflective \phys, with initial position and velocity $(\alpha(0),\dot\alpha(0))\in TM$. For sufficiently small $\ep>0$, there exists a neighborhood $V_{\ep}$ of $(\alpha(0),\dot\alpha(0))$ in $TM$ with the property that for all $(x,v)\in V_{\ep}$, there exists a unique \reflective \phys $\alpha_{x,v}$ such that
$(\alpha_{x,v}(0),\dot{\alpha}_{x,v}(0)) = (x,v)$,
$\alpha_{x,v}$ has reflection times $\Rt_{x,v}$ with $|\Rt_{x,v}| = |\Rt|$,
and, if $\Rt = \{T_1<\dots<T_r\}$ and $\Rt_{x,v} = \{\tilde{T}_1<\dots<\tilde{T}_r\}$, we have
$|\tilde{T}_i-T_i|\le\ep$ for $i=1,\dots,r$.
\end{lemma}
\begin{proof}
The idea is that since a \reflective \phys solves a second-order ODE classically up to reflection times, it is uniquely specified, up until the reflection time, by its initial position and velocity; at reflection times the path experiences a jump in its velocity uniquely specified by the reflection condition \eqref{reflectionitem} in Definition \ref{def:path.global}, which uniquely specifies the path until the next reflection, and so on. We make this idea rigorous below.

Suppose for convenience that $\Rt$ consists of a single time $T_1$,
i.e.\ $\alpha$ reflects just once; let $x^1$ be a boundary defining
function of $Y$ near $\alpha(T_1)$. Let $\ep>0$ be sufficiently small,
so that there exists a neighborhood $W$ of
$(\alpha(T_1),\dot\alpha(T_1-))\in TM\backslash TY$ satisfying the
following technical assumption: if $\beta$ is a $C^2([0,2\ep])$
solution to $D_t\dot{\beta}(t)+\nabla V(\beta(t)) = 0$
in $(0,2\ep)$, with $\beta(0)\in Y$ and $(\beta(0),\dot\beta(0))$
or $(\beta(0),Q\dot\beta(0))$ is in $W$, where $Q:T_YM\to T_YM$ is
the reflection across $Y$, then $\text{sgn }(x^1\circ\beta)'(t)$ is
constant on $(0,2\ep)$. (That is, \reflective \physs starting on $Y$
with initial velocity in $W$ or $Q(W)$ will always move away from the
boundary and will not return in time $2\ep$.)
Such a neighborhood exists for $\ep$ sufficiently small by the non-tangency assumption of $\alpha$ at reflected times.

Given $(x,v)$ near $(\alpha(0),\dot\alpha(0))$, we construct a nearby \reflective \phys $\alpha_{x,v}$ as follows: we note that if $(x,v)$ is sufficiently close to $(\alpha(0),\dot\alpha(0))$ (or any vector if $M$ is complete), then there exists a unique $C^2$ solution on $[0,T_1+\ep]$ to
\[D_t\dot{\tilde{\alpha}}(t)+\nabla V(\tilde\alpha(t)) = 0\text{ in }(0,T_1+\ep),\quad (\tilde{\alpha}(0),\dot{\tilde\alpha}(0)) = (x,v).\]
Note that such a solution, if it intersects $Y$, does not reflect off $Y$. Moreover, if $(x,v)$ is sufficiently close to $(\alpha(0),\dot\alpha(0))$, then the corresponding path $\tilde\alpha$ intersects $Y$ at some time in $[T_1-\ep,T_1+\ep]$; let $\tilde{T}_1$ denote the first such time. Finally, for $(x,v)$ sufficiently close to $(\alpha(0),\dot\alpha(0))$, we can also arrange for $(\tilde\alpha(\tilde{T}_1),\dot{\tilde\alpha}(\tilde{T}_1-))\in W$. We let $V$ be a neighborhood of $(\alpha(0),\dot\alpha(0))$ such that its elements satisfy all of the conditions above.

Then, for $(x,v)\in V$, we construct $\alpha_{x,v}$ as follows. For $t\in [0,\tilde{T}_1]$, we set $\alpha_{x,v}(t) = \tilde{\alpha}(t)$ as above. For $t\in[\tilde{T}_1,T]$, we let $\alpha_{x,v}$ be the unique $C^2([\tilde{T}_1,T])$ solution to
\[D_t\dot\alpha_{x,v}(t) + \nabla V(\alpha_{x,v}(t)) = 0\text{ in }(\tilde{T}_1,T),\quad (\alpha_{x,v}(\tilde{T}_1),\dot\alpha_{x,v}(\tilde{T}_1+)) = (\tilde\alpha(T_1),Q\dot{\tilde\alpha}(\tilde{T}_1-)).\]
Then, by construction, $\alpha_{x,v}$ is a \reflective \phys, with one reflection at $\tilde{T}_1$ satisfying $|\tilde{T}_1-T_1|\le\ep$ (note that $\dot\alpha_{x,v}(\tilde{T}_1-)=\tilde{\dot\alpha}(\tilde{T}_1-)\in W\subset TM\backslash TY$ guarantees the non-tangency condition). This shows the existence for $(x,v)\in V$.

For uniqueness, we note that if $\beta_{x,v}$ is another \reflective \phys satisfying $(\beta_{x,v}(0),\dot\beta_{x,v}(0)) = (x,v)$ with exactly one reflection time $\tau_1$ satisfying $|\tau_1-T_1|\le \ep$, then $\beta_{x,v}$ is $C^2$ on $[0,\tau_1]$, and in particular it must agree with $\alpha_{x,v}$ up to time $T_1-\ep$, after which it continues to agree with $\alpha_{x,v}$ until $\alpha_{x,v}$ hits $Y$, i.e. at time $\tilde{T}_1$. The only way $\beta_{x,v}$ does not agree with $\alpha_{x,v}$ after that is if $\beta_{x,v}$ transmits through $Y$ instead of reflecting across $Y$, i.e. $\dot\beta_{x,v}(\tilde{T}_1+)$ equals $\dot\alpha_{x,v}(\tilde{T}_1-)$ instead of its reflection. However, by the technical assumption made above, this would force $\beta_{x,v}$ to not intersect $Y$ again in $(\tilde{T}_1,\tilde{T}_1+2\ep]$, and in particular it will not reflect at a time within $\ep$ of $T_1$. This forces $\beta_{x,v}$ to reflect at $\tilde{T}_1$, and since there are no other reflections, this means $\beta_{x,v}$ is a $C^2$ solution on $[\tilde{T}_1,T]$ whose value and derivative agrees with those of $\alpha_{x,v}$ at $\tilde{T}_1$, forcing $\beta_{x,v} = \alpha_{x,v}$ on $[\tilde{T}_1,T]$ as well. This gives uniqueness as well.

The case for multiple reflections is similar, by performing the above technical constructions in a neighborhood of each reflection time.

\end{proof}

Given a variation $\alpha(t,\ep)$ along a family of \reflective paths, we can consider the tangent vector field
$$Z(t) =\frac{\pa \alpha}{\pa \ep}\big\rvert_{\ep=0}$$
along $\alpha(t)$. Note that the $\ep$ derivative is only well-defined on $[0,T]\backslash(\Rt\cup\Kt)$; however the one-sided limits $Z(T_j\pm)$ exist for all $T_j\in\Rt\cup\Kt$. We use this notion
to define corresponding tangent spaces $T_\alpha\Omega(M)$,
$T_\alpha\Omega_0(M;p,p'),$ and $T_\alpha\Omega_{\per}(M)$. These
spaces are characterized by our enforcement of the continuity
conditions at the boundary, as follows.

\begin{lemma}\label{lemma:tangentspace}
 Let $\alpha(t,\ep)\in\Omega(M)$ be a family of \reflective paths 
 with reflection time functions $\tilde\Rt$ and kink time functions $\tilde\Kt$, and let $\tilde\Rt\cup\tilde\Kt = \{T_i(\ep)\,:\,1\le i\le m\}$. 
Then, for each $1\le i\le m$, the variation vector field $Z=\pa\alpha/\pa \ep$ satisfies the jump condition 
\begin{equation}
\label{eq:adm-cont}
T_i'(0)\dot\alpha(T_i-) + Z(T_i-) = T_i'(0)\dot\alpha(T_i+) + Z(T_i+).
\end{equation}
If in addition we have $T_i\in\Rt$, then we have the additional condition that
\begin{equation}
\label{eq:adm-ref}
T_i'(0)\,  \dot \alpha_\nrml(T_i-) +Z_\nrml(T_i-) =  T_i'(0)\,  \dot \alpha_\nrml(T_i+) +Z_\nrml(T_i+)=0.
\end{equation}
\end{lemma}
In particular, if $\alpha$ is $\mathcal{C}^1$ at $T_i$, then $Z(T_i-)=Z(T_i+)$, i.e.\ at such times we may view $Z$ as being well-defined and continuous at $T_i$.
\begin{proof}
For each $i$ we have the continuity equation
\[\alpha(T_i(\ep)-,\ep) = \alpha(T_i(\ep)+,\ep).\]
Differentiating in $\ep$ and evaluating at $\ep=0$ yields \eqref{eq:adm-cont}.
If, in addition, $T_i\in\Rt$, then as usual letting $\alpha_1(t)$
denote the signed distance from the boundary (first component in Fermi coordinates),
for all $\ep$,
$$\alpha_1(T_i(\ep)\pm,\ep)=0$$ (since the path is in the boundary at time
$T_i(\ep)$).
Differentiating
in $\ep$ and evaluating at $\ep=0$ yields \eqref{eq:adm-ref}.
\end{proof}

\begin{lemma}\label{lemma:integration}
Let $\alpha\in\Omega(M)$ be a \reflective path with reflections and kinks at $\Rt$ and $\Kt$, respectively, with $\Rt\cup\Kt = \{T_i\,:\,1\le i\le m\}$.  Let $Z$ be a vector field along $\alpha$ defined on $[0,T]\backslash(\Rt\cup\Kt)$
such that $Z|_{(T_i,T_{i+1})}$ extends smoothly to $[T_i,T_{i+1}]$ for all $0\le i\le m$, and that for some $\mu_1,\dots, \mu_m \in \RR$,
\begin{equation}
\label{eq:muadm-cont}
\mu_i\dot\alpha(T_i-) + Z(T_i-) = \mu_i\dot\alpha(T_i+) + Z(T_i+) 
\end{equation}
and
\begin{equation}
\label{eq:muadm-ref}
\mu_i\dot\alpha_\nrml(T_i-) + Z_\nrml(T_i-) = \mu_i\dot\alpha_\nrml(T_i+) + Z_\nrml(T_i+) = 0\quad\text{if }T_i\in\Rt.
\end{equation}
Then there is a variation $\alpha(t,\ep) \in\Omega(M)$ of $\alpha$ with $Z=\pa
\alpha/\pa \ep\rvert_{\ep=0}$ and with $T_i'(0) =\mu_i$.
  \end{lemma}

\begin{remark}
\eqref{eq:muadm-cont} can be rewritten as
\[\jump Z(T_i) = -\mu_i\jump\dot\alpha(T_i).\]
Moreover, at reflection times $T_i\in\Rt$, the two conditions \eqref{eq:muadm-cont} and \eqref{eq:muadm-ref} can be rephrased as coupled jump conditions on the tangential and normal components, via
\[Z_\nrml(T_i\pm) = -\mu_i\dot\alpha_\nrml(T_i\pm),\quad \jump Z_{\tgt}(T_i) = -\mu_i\jump\dot\alpha_{\tgt}(T_i).\]
Finally, given the corresponding endpoint conditions, we may obtain variations $\alpha(t,\ep) \in \Omega_0(M; p,p')$ or $\Omega_{\per}(M)$ from conditions 
\eqref{eq:muadm-cont} and \eqref{eq:muadm-ref}	as in the lemma.
\end{remark}
	
\begin{proof}
We construct such a variation explicitly. The idea is that we consider \emph{some} variation whose derivative is $Z$, and then correct for reflection/kink conditions.

For each $T_i\in\Rt\cup\Kt$, choose a neighborhood $I_i$ in $t$ such that $I_i\cap(\Rt\cup\Kt) = \{T_i\}$, and $\alpha(I_i)$ is contained
in an open set trivializable by local coordinates, where if $T_i\in\Rt$, then the local coordinates are Fermi coordinates oriented
  so that
  $\alpha_1 \geq 0$ for $t \in I_i$. Let $I = \cup I_i$. For $t\in[0,T]\backslash I$, we define
\[\alpha(t,\ep) := \exp_{\alpha(t)}(\ep Z(t)),\]
where $\exp$ is the exponential map with respect to some Riemannian
metric (e.g.\ the metric $g$ on $M$), smoothly extended across
  the boundary in the case where $Y=\pa M$ locally.
Note that this is well-defined and smooth on $([0,T]\backslash I)\times(-\ep_0,\ep_0)$ for sufficiently small $\ep_0>0$. Moreover,
\[\frac{\partial\alpha}{\partial\ep}\rvert_{\ep=0}(t) = Z(t)\text{ on }[0,T]\backslash I.\]
For $t\in I_i$, $1\le i\le m$, we define $\alpha$ via local coordinates $(x^1,\dots,x^n)$ (which are local Fermi coordinates if $T_i\in\Rt$), i.e.\ we define the values of $\alpha_j(t,\ep) = x^j\circ\alpha(t,\ep)$ for $1\le j\le n$. Write $\alpha^-=\alpha|_{(T_{i-1},T_i)}$ and $\alpha^+=\alpha|_{(T_i,T_{i+1})}$, define $Z^{\pm}$ similarly, and extend $\alpha^{\pm}$, $Z^{\pm}$ smoothly in a neighborhood of $T_i$. Set $T_i(\ep) = T_i + \ep\mu_i$ so that $T_i'(0)=\mu_i$, and for $t\in I_i$, let
\[\alpha_j(t,\ep) = \begin{cases} (\exp_{\alpha^-(t)}(\epsilon
    Z^-(t)))_j + r_j^-(\ep)\varphi(t) & t\le T_i(\ep) \\
    (\exp_{\alpha^+(t)}(\epsilon Z^{+}(t)))_j + r_j^+(\ep)\varphi(t) &
    t\ge T_i(\ep)\end{cases}\] 
where $\varphi(t)\in C_c^\infty(I_i)$ is identically equal to $1$ in a small neighborhood of $T_i$. The vector-valued functions $r^{\pm}(\ep)$ are defined depending on whether $T_i\in\Rt$ or $T_i\in\Kt$: if $T_i\in\Rt$, set
\begin{align*}
r_1^-(\ep) &= -(\exp_{\alpha^-(T_i(\ep))}(\epsilon Z^-(T_i(\ep))))_1, \quad 
r_j^-(\ep) = 0 \text{ for }j\ge 2, \\
r_1^+(\ep) &= -(\exp_{\alpha^+(T_i(\ep))}(\ep Z^+(T_i(\ep))))_1, \\
r_j^+(\ep) &= (\exp_{\alpha^-(T_i(\ep))}(\ep Z^-(T_i(\ep))))_j - (\exp_{\alpha^+(T_i(\ep))}(\ep Z^+(T_i(\ep))))_j\text{ for }j\ge 2.
\end{align*}
If $T_i\in\Kt$, set $r^-(\ep) = 0$ and 
\[r_j^+(\ep) = (\exp_{\alpha^-(T_i(\ep))}(\ep Z^-(T_i(\ep))))_j - (\exp_{\alpha^+(T_i(\ep))}(\ep Z^+(T_i(\ep))))_j\text{ for }1\le j\le n.\]

We first verify that this construction produces paths
$\alpha(\cdot,\ep)$ which are \reflective paths for $\ep$ sufficiently
small and that the construction makes sense in the case $Y=\pa
  M$, i.e., that the constructed family of paths stays in $M$ rather than passing
  into an extension across the boundary. By construction, $\alpha(\cdot,\ep)$ is continuous on $I_i$ and
is smooth on $I_i\backslash\{T_i(\ep)\}$. Furthermore, if $T_i\in\Rt$,
then taking $\supp \varphi$ sufficiently small, nonnegativity of
$\alpha_1$ and nontangency allow us to ensure
$$
\sgn \dot \alpha_1^\pm=\pm 1 \quad \text{ on } \supp \varphi.
$$
Then taking $\ep>0$ sufficiently small ensures that the same holds for
the varied path, i.e., 
$$
\sgn \dot \alpha_1^\pm(t,\ep)=\pm 1 \quad \text{ on } \supp \varphi.
$$
For small $\ep$ it is also the case that $\alpha_1 \neq 0$ for $t \in I_i\backslash  \supp
\varphi$.  Since $\alpha_1^\pm (T_i(\ep), \ep)=0$ by construction, this
shows that $\alpha_1^\pm$ remains nonnegative for $t\in I_i$, and vanishes only
at $t=T_i(\ep).$  In particular, then, the family of paths
  remains in $M$ even when $Y=\pa M$ locally.
Finally, we have
\[\dot\alpha_1^{\pm}(T_i(\ep),\ep) = \dot\alpha_1^{\pm}(T_i(\ep)) + \ep Z_1^{\pm}(T_i(\ep)) + r_1^{\pm}(\ep)\varphi'(T_i(\ep)) + O(\ep^2),\]
and this is nonzero for $\ep$ sufficiently small since $\dot\alpha_1^{\pm}(T_i)\ne 0$ and $r_1^{\pm}(\ep)\to 0$ as $\ep\to 0$. Hence, $\alpha(t,\ep)$ intersects $Y$ transversely at $t=T_i(\ep)$.
Thus, $\alpha(t,\ep)$ is a family of reflected paths.

Finally, need to check that the first derivative of $\alpha$ is in
fact $Z$.  We clearly have $\frac{\partial\alpha}{\partial\ep}\rvert_{\ep=0} = Z$ on $[0,T]\backslash I$, while on $I_i$ we have
\[\frac{\partial\alpha}{\partial\ep}\rvert_{\ep=0} = \begin{cases} Z^-(t) + (r^-)'(0)\varphi(t) & t\in I_i, t<T_i \\ Z^+(t) + (r^+)'(0)\varphi(t) & t\in I_i, t>T_i\end{cases}.\]
Thus it suffices to show that $(r^{\pm})'(0) = 0$. If $T_i\in\Rt$, we have
\[(r_1^{\pm})'(0) = -(\mu_i\dot\alpha_1^{\pm}(T_0) + Z_1^{\pm}(T_0)) = 0,\]
with the last equality following from \eqref{eq:muadm-ref}, while for $j\ge 2$, by \eqref{eq:muadm-cont} we have
\[(r_j^-)'(0)=0,\quad (r_j^+)'(0) = (\mu_i(\dot\alpha_j^-(T_0)-\dot\alpha_j^+(T_0)) + (Z_j^-(T_0)-Z_j^+(T_0))) = 0.\]
It follows that $(r^{\pm})'(0)=0$, as desired. Similar calculations show that $(r^{\pm})'(0)=0$ in the case that $T_i\in\Kt$ as well.
\end{proof}
\begin{remark}
An analogous statement and proof holds for constructing two-parameter family of variations with admissible pairs of variation vector fields $Z$ and $W$ (say satisfying \eqref{eq:muadm-cont} and \eqref{eq:muadm-ref} with $\mu_i$, $\nu_i$), by defining $T_i(\epsilon,\delta)= T_i+\epsilon\mu_i+\delta\nu_i$, $\alpha(t,\epsilon,\delta) = \exp_{\alpha(t)}(\epsilon Z(t)+\delta W(t))$ away from reflections and kinks, and correcting analogously near the reflections/kinks. We omit the proof for brevity.
\end{remark}

\begin{corollary}
\label{cor:suff}
For any choice of $V_i\in T_{\alpha(T_i)}M$ such that $V_i\in T_{\alpha(T_i)}Y$ when $T_i\in\Rt$, there exists a variation $\alpha(t,\ep)\in\Omega(M)$ such that $\frac{\partial\alpha}{\partial\ep}\rvert_{\ep=0}$ is continuous for all $t$, with
\[\frac{\partial\alpha}{\partial\ep}\rvert_{\ep=0}(T_i) = V_i.\]
Moreover, such a variation can be chosen so that the corresponding reflection times $T_i(\ep)$ satisfy $T_i'(0)=0$.
\end{corollary}
\begin{proof}
Let $Z$ be \emph{any} smooth vector field with $Z(T_i)=V_i$. Then $Z$ satisfies \eqref{eq:muadm-cont} and \eqref{eq:muadm-ref} with all $\mu_i$ equal to $0$.
\end{proof}
\begin{corollary}
\label{cor:sufft}
Let $\mu_1,\dots,\mu_m$ be any collection of numbers. Then there exists a variation $\alpha(t,\ep)\in\Omega(M)$ where the corresponding reflection times $T_i(\ep)$ satisfy $T_i'(0) = \mu_i$.
\end{corollary}
\begin{proof}
Let $Z$ be smooth on $[0,T]\backslash(\Rt\cup\Kt)$ with limits at $T_i\in\Rt$ determined by
\begin{align*}
Z_\nrml(T_i-) &= -\mu_i\dot\alpha_\nrml(T_i-), \quad  Z_\tgt(T_i-)\text{ arbitrary},\\
Z(T_i+) &= Z(T_i-) - \mu_i(\dot\alpha(T_i+)-\dot\alpha(T_i-)),
\end{align*}
and with limits at $T_i\in\Kt$ determined by
\[Z(T_i-)\text{ arbitrary},\quad Z(T_i+) = Z(T_i-) - \mu_i(\dot\alpha(T_i+)-\dot\alpha(T_i-)).\]
Then $Z$ satisfies \eqref{eq:muadm-cont} and \eqref{eq:muadm-ref} with the prescribed values of $\mu_i$.
\end{proof}
\begin{corollary}
\label{cor:physz}
Let $\alpha(t)$ be a \reflective\phys. Then a piecewise smooth vector field $Z$ along $\alpha$ is a variation vector field if and only if:
\begin{itemize}
\item $Z$ is smooth on $[0,T]\backslash(\Rt\cup\Kt)$,
\item At reflection times $T_i\in\Rt$, we have
\[\jump Z_\tgt(T_i) = 0,\quad \ol Z_\nrml(T_i) = 0\]
(i.e.\ $\jump Z$, resp.\ $\ol Z$, is normal, resp.\ tangent, to the hypersurface).
\item At kink times $T_i\in\Kt$, we have
\[\jump Z(T_i) = 0\]
(i.e.\ $Z$ is continuous at kink times).
\end{itemize}
\end{corollary}
\begin{proof}
This follows by rewriting the conditions in Lemmas \ref{lemma:tangentspace} and \ref{lemma:integration}, using that in a \reflective\phys we have the conditions $\ol{\dot\alpha_\nrml (T_i)} = 0$ and $\jump\dot\alpha_\tgt (T_i) = 0$ if $T_i\in\Rt$ and $\jump\dot\alpha(T_i)=0$ if $T_i\in\Kt$.
\end{proof}

Let $\alpha$ be a \reflective path, with reflection times $\Rt$. By
Remark \ref{rem:pathrem}, $\alpha$ admits a minimal kink time set
$\Kt_{min}$. 
Let $\mathcal{V}(\alpha)$ denote the set of vector fields along $\alpha$. By Lemmas~\ref{lemma:tangentspace} and \ref{lemma:integration}, we may identify the tangent spaces
 to our various path spaces as follows:
\begin{align*}
T_\alpha \Omega(M)& = \left\{W\in\mathcal{V}(\alpha): \exists~\Kt\supset\Kt_{min} \text{ s.t. } W \text{ is smooth on $[0,T]\backslash(\Rt\cup\Kt)$}\right.\\
&\left.\text{satisfying jump conditions \eqref{eq:muadm-cont} and \eqref{eq:muadm-ref} at }\Rt\cup\Kt\right\},\\
T_\alpha\Omega_0(M;p,p') & = \{W\in T_\alpha\Omega(M): W(0) =
                           W(T)=0\}\\
  T_\alpha\Omega_{\per}(M) & = \{W\in T_\alpha\Omega(M): W(0) = W(T)\}.
\end{align*}
If $\alpha$ is a periodic \reflective path with $\alpha(0)=\alpha(T)=p$, then
\[T_\alpha\Omega_0(M;p,p)\subset T_\alpha\Omega_{\per}(M) \subset T_\alpha\Omega(M),\]
while if $\alpha$ is a (\reflective) path from $p$ to $p'$ (with $p$ not necessarily equal to $p'$) then
\[T_\alpha\Omega_0(M;p,p')\subset T_\alpha\Omega(M).\]
\begin{remark}
    Note that the definition of $T_\alpha \Omega(M)$
    allows for freedom in the choice of $\Kt$: in addition to points on $\alpha$ that fail to be $\Ct$, which by definition occur at times in $\Kt_{min} \cup \Rt$, we may always add ``fictitious''
    extra kink points (i.e.\ $\Kt\backslash\Kt_{min}$) where we do not enforce continuity of
    derivatives of variation vector fields; note at those points that $\alpha$ is smooth, and hence the jump condition at those points reduces to the condition of continuity.
		Consequently an authentic
    kink is allowed to develop in the varied paths at such points.
		
		Moreover, the space $T_\alpha\Omega(M)$ (and hence $T_\alpha\Omega_0(M;p,p')$ and $T_\alpha\Omega_{\per}(M)$) is in fact a vector space. Indeed, given $Z_1, Z_2\in T_\alpha\Omega(M)$ with kink times $\Kt_1$ and $\Kt_2$, we may view both of them as having kink times at $\Kt_1\cup\Kt_2$ by adding ``fictitious'' kink points, and hence it is clear that $Z_1+Z_2$ satisfies the jump conditions \eqref{eq:muadm-cont} and \eqref{eq:muadm-ref} at $\Rt\cup(\Kt_1\cup\Kt_2)$.
  \end{remark}

\section{Actions along paths and first variations}
\label{sec:1var}

In the following sections, we shall consider the action functional on a family of
\reflective paths $\{\alpha(\cdot,\epsilon):
\ep\in(-\ep_0,\ep_0)\}\subset \Omega_0(M;p,p')$. 
The action functional is given by

\begin{equation}
\label{functional}
\begin{aligned}
J[\alpha(\cdot,\epsilon)] &= \int_0^T{\frac{1}{2}|\dot \alpha(t,\epsilon)|_g^2-V(\alpha(t,\epsilon))\,dt} \\
&:= \sum_{i=0}^m \int_{T_i(\ep)}^{T_{i+1}(\ep)}{\frac{1}{2}|\dot \alpha(t,\epsilon)|_g^2-V(\alpha(t,\epsilon))\,dt},
\end{aligned}
\end{equation}
where as usual we take $T_0(\ep) = 0$ and $T_{m+1}(\ep) = T$.

\subsection{First variations}
The following lemma gives the derivative of the actions \eqref{functional}.
\begin{lemma}
    The derivative $\frac{d}{d\epsilon}J[\alpha(\cdot,\epsilon)]$ is given by 
    \begin{equation}
    \label{eq:var1_broken}
-\int_0^T \bigang{\cov\dot \alpha +\nabla V(\alpha), \frac{\partial \alpha}{\partial\epsilon}(t,\ep)}\,dt + \sum_{i=1}^m \bigang{-\jump\dot \alpha(T_i(\ep),\epsilon),\ol{\frac{\partial \alpha}{\partial\epsilon}}(T_i(\ep),\ep)}.
\end{equation}
\end{lemma}
\begin{proof}

For each term in the sum in the second line of equation \eqref{functional}, we differentiate in $\ep$ to obtain
\begin{align*}
&\frac{d}{d\ep}\left(\int_{T_i(\ep)}^{T_{i+1}(\ep)}{\frac{1}{2}|\dot \alpha(t,\epsilon)|_g^2-V(\alpha(t,\epsilon))\,dt}\right)
= \int_{T_i(\ep)}^{T_{i+1}(\ep)}\bigang{\dot \alpha,\cov\frac{\partial{ \alpha}}{\partial\epsilon}} - \bigang{\nabla V(\alpha),\frac{\partial{\alpha}}{\partial\epsilon}}\,dt \\
&+T_{i+1}'(\ep)\left(\frac{1}{2}|\dot\alpha(T_{i+1}(\ep)-,\ep)|^2_g - V(\alpha(T_{i+1}(\ep),\ep))\right)-T_i'(\ep)\left(\frac{1}{2}|\dot\alpha(T_i(\ep)+,\ep)|^2_g - V(\alpha(T_i(\ep),\ep))\right),
\end{align*}
where the integral term follows using the metric compatibility and the torsion-free property of the Levi-Civita connection (the latter yields $\frac{\partial}{\partial\ep}\dot\alpha=\cov\frac{\partial\alpha}{\partial\ep}$); note that we can evaluate $\alpha$ at $T_i(\ep)$ since $\alpha$ is continuous. We can then use the metric compatibility of $\nabla$ to integrate by parts, yielding
\begin{align*}
&\int_{T_i(\ep)}^{T_{i+1}(\ep)}\bigang{\dot \alpha,\cov\frac{\partial{ \alpha}}{\partial\epsilon}} - \bigang{\nabla V(\alpha),\frac{\partial{\alpha}}{\partial\epsilon}}\,dt \\
&= -\int_{T_i(\ep)}^{T_{i+1}(\ep)}{\bigang{\cov \dot \alpha+\nabla V(\alpha),\frac{\partial \alpha}{\partial\epsilon}}\,dt} + \left[\bigang{\dot\alpha,\frac{\partial\alpha}{\partial\ep}}\right]\Big\rvert_{T_i(\ep)+}^{T_{i+1}(\ep)-}.
\end{align*}
Adding the terms together, noting that the boundary terms involving $V$ cancel, $T_0'(\ep) = 0 = T_{m+1}'(0)$, and that $\frac{\partial\alpha}{\partial\ep}$ vanishes at $t=0,T$ by assumption of fixed endpoints, we obtain
\begin{equation}
\label{1var}
\begin{aligned}
\frac{d}{d\epsilon}(J[\alpha(\cdot,\epsilon)]) &= -\int_0^T{\bigang{\cov \dot \alpha+\nabla V(\alpha),\frac{\partial \alpha}{\partial\epsilon}}\,dt} \\
&-\sum_{i=1}^m \left(\frac{1}{2}T_i'(\ep)\jump\left(|\dot\alpha|^2_g\right) + \jump\left(\bigang{\dot\alpha,\frac{\partial\alpha}{\partial\ep}}\right)\right)\Big\rvert_{t=T_i(\ep)}.
\end{aligned}
\end{equation}
We now rewrite each term in the second line of \eqref{1var}. Fix $1\le i\le m$, and for convenience, let $\dot\alpha^{\pm}$ and $\frac{\partial\alpha}{\partial\ep}^{\pm}$ denote the values of each quantity at $T_i(\ep)\pm$. 
Using the algebraic identity $\langle a,b\rangle-\langle c,d\rangle = \frac{1}{2}\langle a+c,b-d\rangle + \frac{1}{2}\langle a-c,b+d\rangle$, 
we have
\begin{align*}
\jump\left(\bigang{\dot\alpha,\frac{\partial\alpha}{\partial\ep}}\right) &= \bigang{\dot\alpha^+,\frac{\partial\alpha}{\partial\ep}^+} - \bigang{\dot\alpha^-,\frac{\partial\alpha}{\partial\ep}^-} \\
&=\frac{1}{2}\bigang{(\dot \alpha^+ + \dot \alpha^-),\left(\frac{\partial \alpha^+}{\partial\epsilon} - \frac{\partial \alpha^-}{\partial\epsilon}\right)} 
+\frac{1}{2}\bigang{(\dot \alpha^+-\dot \alpha^-),\left(\frac{\partial \alpha^+}{\partial\epsilon} + \frac{\partial \alpha^-}{\partial\epsilon}\right)} \\
&=\bigang{\ol{\dot\alpha},\jump\frac{\partial\alpha}{\partial\ep}} + \bigang{\jump\dot\alpha,\ol{\frac{\partial\alpha}{\partial\ep}}}.
\end{align*}
Furthermore, from Lemma \ref{lemma:tangentspace} we have
\
\begin{equation}
\label{eq:relation}
T_i'(\ep)\dot \alpha^- + \frac{\partial \alpha^-}{\partial\epsilon} = T_i'(\ep)\dot \alpha^+ + \frac{\partial \alpha^+}{\partial\epsilon}\implies \jump\frac{\partial\alpha}{\partial\ep} = -T_i'(\ep)\jump\dot\alpha.
\end{equation}
Thus
\[\bigang{\ol{\dot\alpha},\jump\frac{\partial\alpha}{\partial\ep}} = -T_i'(\ep)\bigang{\ol{\dot\alpha},\jump\dot\alpha} = -\frac{1}{2}T_i'(\ep)\jump(|\dot\alpha|^2_g),\]
and hence
\begin{equation}
    \label{eq:sub_idd}
\jump\left(\bigang{\dot\alpha,\frac{\partial\alpha}{\partial\ep}}\right) = \bigang{\ol{\dot\alpha},\jump\frac{\partial\alpha}{\partial\ep}} + \bigang{\jump\dot\alpha,\ol{\frac{\partial\alpha}{\partial\ep}}} = -\frac{1}{2}T_i'(\ep)\jump(|\dot\alpha|^2_g) + \bigang{\jump\dot\alpha,\ol{\frac{\partial\alpha}{\partial\ep}}}.
\end{equation}
Substituting equation \eqref{eq:sub_idd} into equation \eqref{1var} yields \eqref{eq:var1_broken}, as desired.
\end{proof}

\subsection{Critical points of first variations}
We prove the following lemma in this subsection:
\begin{lemma}
    \label{lemma:reflective_path}
For any variation $\alpha(t,\epsilon)\in\Omega_0(M;p,p')$ with
$\alpha(t,0) = \alpha(t)$, $\alpha(t)$ is a critical point of
$J[\alpha(\cdot,\ep)]$, in the sense that
$\frac{d}{d\epsilon}\Big|_{\epsilon=0}(J[\alpha(\cdot,\epsilon)]) = 0$
holds, if and only if $\alpha(t)$ is a \reflective \phys.
\end{lemma}

\begin{proof}
Suppose that
$\frac{d}{d\epsilon}\Big|_{\epsilon=0}(J[\alpha(\cdot,\epsilon)])=0$
for all variations in $\Omega_0(M;p,p')$ with $\alpha(t,0) =
\alpha(t)$. We first consider variations where $Z =\frac{\partial
  \alpha}{\partial\epsilon}\rvert_{\ep=0}$ vanishes at all times of reflection and kinks $t=T_i(0)\in\Rt\cup\Kt$, with the corresponding time functions satisfying $T_i'(0)=0$; such variations are possible by Corollary \ref{cor:suff}. Then all of the boundary terms in \eqref{eq:var1_broken} vanish, and we obtain
\[0 = \frac{d}{d\epsilon}\Big|_{\epsilon=0}(J[\alpha(\cdot,\epsilon)]) = -\int_0^T{\bigang{\cov\dot \alpha(t)+\nabla V(\alpha(t)),Z(t)}\,dt}.\]
We note that we can arrange for $Z(t)$ to be any arbitrary smooth vector-valued function which vanishes at $t=0$, $t=T$, and all $t=T_i\in\Rt\cup\Kt$, and such vector fields are dense in $L^2$. It follows that $\alpha(t)$ must satisfy
\begin{equation}
    \label{eq:ham_path}
    D_t\dot \alpha(t)+\nabla V(\alpha(t)) = 0\text{ on }[0,T]\backslash(\Rt\cup\Kt),
\end{equation}
i.e.\ condition \eqref{geodesicitem} in Definition \ref{def:path.global}. Hence, for a path $\alpha(t)$ where $J(\alpha(\cdot,\epsilon))$ is stationary at $\epsilon=0$, \eqref{eq:var1_broken} reduces to
\begin{equation}
\label{1var3}
\begin{split}
    0=\frac{d}{d\epsilon}\Big|_{\epsilon=0}(J[\alpha(\cdot,\epsilon)]) =
     \sum_{i=1}^m\bigang{-\jump\dot \alpha(T_i),\overline{Z}(T_i) }
\end{split}
\end{equation}
for all variations $\alpha(t,\epsilon)\in \Omega_0(M;p,p')$ with $\alpha(t,0)=\alpha(t)$.

We then consider variations $\alpha(t,\ep)$ such that $Z(T_i)=0$ and
$T_i'(0)=0$ for all reflection times
$T_i\in\Rt$, 
but $Z$ does not necessarily vanish at the kink times. For such 
variations, equation \eqref{1var3} becomes
\[   0=\frac{d}{d\epsilon}\Big|_{\epsilon=0}(J[\alpha(\cdot,\epsilon)]) =
\sum_{T_i\in\Kt}\bigang{-\jump\dot \alpha(T_i),\overline{Z}(T_i) }.\]
By Corollary \ref{cor:suff}, we may take $Z$ to be continuous, with arbitrary values at $T_i\in\Kt$. Thus  $\jump\dot \alpha(T_i) = 0$ for all $T_i\in\Kt$, which is condition \eqref{kinkitem} in Definition \ref{def:path.global}. It follows that, for any variation, equation \eqref{1var3} further reduces to 
\begin{equation}
\label{1var3-1}
\begin{split}
    0=\frac{d}{d\epsilon}\Big|_{\epsilon=0}(J[\alpha(\cdot,\epsilon)]) =
    \sum_{T_i\in\Rt}\bigang{-\jump\dot \alpha(T_i),\overline{Z}(T_i) }.
\end{split}
\end{equation}
Each term in the above sum can be rewritten as
\begin{align*}
\bigang{-\jump\dot \alpha(T_i),\overline{Z}(T_i) } &= \bigang{-\jump \dot \alpha_{\tgt}(T_i), \overline{Z_{\tgt}}(T_i)} + \bigang{-\jump \dot\alpha_{\nrml}(T_i), \overline{Z_{\nrml}}} \\
&=\bigang{-\jump \dot \alpha_{\tgt}(T_i), \overline{Z_{\tgt}}(T_i)} - \frac{1}{2}T_i'(0)\jump(|\dot\alpha|_g^2)\rvert_{t=T_i},
\end{align*}
since $\ol{Z_{\nrml}}(T_i) = -T_i'(0)\ol{\dot\alpha(T_i)}$ by equation \eqref{eq:adm-ref}. Hence equation \eqref{1var3-1} can be rewritten as
\begin{equation}
\label{1var3-2}
    0 = -\sum_{T_i\in\Rt} \left(\bigang{\jump \dot \alpha_{\tgt}(T_i), \overline{Z_{\tgt}}(T_i)} + \frac{1}{2}T_i'(0)\jump(|\dot\alpha|_g^2)\rvert_{t=T_i}\right).
\end{equation}
We now consider variations where $T_i'(0)=0$ for $T_i\in\Rt$, in which case by Corollary \ref{cor:suff} each $\ol{Z_{\tgt}}(T_i)$ can be chosen to be any arbitrary vector tangent to $Y$. Applying Equation \eqref{1var3-2} with $\ol{Z_{\tgt}}(T_i)$ attaining arbitrary tangent values, it follows that
$$\jump \dot \alpha_{\tgt}(T_i)=0$$
for all $i$, which is part of condition \eqref{reflectionitem} in Definition \ref{def:path.global}.
Finally, \eqref{1var3-1} now reduces to
\begin{equation}
\label{1var4}
0=\frac{d}{d\epsilon}\Big|_{\epsilon=0}(J[\alpha(\cdot,\epsilon)]) = -\sum_{T_i\in\Rt}\frac{1}{2}T_i'(0)\jump(|\dot\alpha|_g^2)|_{t=T_i} =-\frac{1}{2}\sum_{T_i\in\Rt}T_i'(0)(|\dot \alpha_{\perp}(T_i+)|_g^2-|\dot \alpha_{\perp}(T_i-)|_g^2)
\end{equation}
(the last equality following since $\jump(|\dot\alpha|_g^2) = \jump(|\dot\alpha_\tgt|_g^2+|\dot\alpha_\nrml|_g^2) = \jump(|\dot\alpha_\nrml|_g^2)$ since we now know that $\jump\dot\alpha_\tgt = 0$). By Corollary \ref{cor:sufft}, we can find variations with arbitrary values of $T_i'(0)$, from which we conclude that
\[|\dot \alpha_\perp(T_i+)|_g^2 - |\dot \alpha_\perp(T_i-)|_g^2 = 0\]
for each $i$. 
Finally, since $\alpha_1^{\pm}\ge 0$ with $\alpha_1^{\pm}(T_i)
= 0$, it follows that $\pm \dot \alpha_1^{\pm}(T_i)\ge 0$. For the
above condition to hold, it must be the case that
	$\dot
\alpha_\perp^-(T_i) + \dot \alpha_\perp^+(T_i) = 0$, which corresponds
to the remaining part of condition \eqref{reflectionitem} in Definition
\ref{def:path.global}. Therefore, we conclude that if
$J(\alpha(\cdot,\epsilon))$ is stationary for any variation
$\alpha(t,\epsilon)$ with $\alpha(t,0)=\alpha(t)$, then $\alpha(t)$
must be a \reflective \phys.

Conversely, suppose $\alpha(t)=\alpha(t,0)$ is a \reflective \phys. 
Then the integral term in \eqref{eq:var1_broken} vanishes by condition
\eqref{geodesicitem} in Definition \ref{def:path.global}, while the
boundary terms over $T_i\in\Kt$ vanish by condition
\eqref{kinkitem}. By condition \eqref{reflectionitem}, at $T_i\in\Rt$, 
$\jump\dot\alpha(T_i)$ is normal to the boundary, while by
\eqref{eq:adm-ref}, $\ol{Z_\perp}(T_i) =
-T_i'(0)\ol{\dot\alpha_\perp}(T_i) = 0$ by condition
\eqref{reflectionitem}; hence the pairing vanishes for each
$T_i$. This gives
$\frac{d}{d\epsilon}\Big|_{\epsilon=0}(J[\alpha(\cdot,\epsilon)]) =
0$, as desired. 
\end{proof}

\section{The Hessian of the action at a \reflective \phys}
\label{sec:2var}

\subsection{Second variations of \reflective \physs}
Now we take $\alpha(t,\epsilon, \delta)\in\Omega_0(M,p,p')$ to be a two-parameter
variation with fixed endpoints of a \reflective \phys
$\alpha(t)$.  Let
\begin{gather*}
    \pa_\ep \alpha(t,\ep,\delta) =Z(t, \ep,\delta),\quad \pa_\delta \alpha(t,\ep,\delta) =W(t, \ep,\delta)\\
    \pa_\ep \alpha(t,0,0) =Z,\quad \pa_\delta \alpha(t,0,0) =W
\end{gather*}
so that (using our jump and average notation from above), the first variation \eqref{eq:var1_broken} now reads
\begin{equation}\label{var1}
    \begin{aligned}
    \frac{d}{d\epsilon}(J[\alpha(\cdot,\epsilon,\delta)]) &= -\int_0^T\bigang{\cov\dot \alpha+\nabla V(\alpha), Z(t,\ep,\delta)}\,dt\\
		&+ \sum_{i=1}^n\bigang{-\jump\dot\alpha(T_i(\ep,\delta),\ep,\delta),\ol{Z}(T_i(\ep,\delta),\ep,\delta)}.
\end{aligned}
\end{equation}

Under the hypothesis that $\alpha(t)=\alpha(t,0,0)$ is a reflected \phys, we will now apply $\pa/\pa\delta$ to the various terms in \eqref{var1} and evaluate at $\delta=0$ to find the second variation.

We now examine the second variation $\frac{\partial^2}{\partial\epsilon \partial\delta}(J[\alpha(\cdot,\epsilon,\delta)])\rvert_{\epsilon=0,\delta=0}$. We split
$$
\frac{\partial^2}{\partial\epsilon \partial\delta}(J[\alpha(\cdot,\epsilon,\delta)])\rvert_{\epsilon=0,\delta=0} := J''_\circ + J''_\pa
$$
where $J''_\circ$, respectively $J''_\partial$, denote the $\delta$
derivative (evaluated at $0$) falling on the integral term in the
first line of \eqref{var1} (the ``interior'' term) and the derivative
falling on the second line (``boundary'' term). 

Differentiating the integral term and evaluating at
$\epsilon=0,\delta=0$ yields
\begin{equation}
    \label{eq:1var_1}
J''_\circ=    -\int_0^T\bigang{D_{\delta} D_t\dot \alpha + (\nabla^2 V)  W,Z}\,dt;
\end{equation}
note that there are no boundary terms arising from differentiating $T_i(\ep,\delta)$ since the integrand $\frac{D}{dt}\dot
\alpha+\nabla V(\alpha)$ equals zero by assumption of $\alpha$ being a \reflective \phys. (We recall that by assumption $\nabla^2
  V$ may have no worse than jump discontinuities across $Y$, so by our
  assumption that $\dot \alpha$ is transverse to $Y$, we may
  legitimately differentiate inside the integral by the Dominated
  Convergence Theorem.)
We now note that
$$
D_{\delta} D_t \dot \alpha =D_t D_{\delta} \dot \alpha+ R\big( \dot\alpha, \frac{\pa \alpha}{\pa \delta}\big) \dot \alpha =D_t^2 W+ R\big( \dot\alpha, W\big) \dot \alpha.
$$
Hence this integral term becomes the standard interior Jacobi equation term
\begin{equation}
J''_\circ=    -\int_0^T\bigang{D^2_t W + R(\dot\alpha, W) \dot \alpha+ (\nabla^2 V) W, Z} \,dt.
\end{equation}

Now we consider the boundary term $J''_\pa$.
Noting that $\pa_\delta\ang{\bullet, \bullet}=\nabla_W\ang{\bullet, \bullet}$, and that the covariant derivative may be brought inside the inner product by compatibility of the connection, we see that each term in the sum differentiates to 
\begin{equation}
  \begin{aligned}
    \label{eq:1var_2}
		& \frac{\partial}{\partial\delta}\left(\bigang{-\jump\dot\alpha(T_i(\ep,\delta),\ep,\delta),\ol{Z}(T_i(\ep,\delta),\ep,\delta)}\right)\\
		&=
    \underbrace{\bigang{-\pa_\delta T_i\jump D_t \dot\alpha, \overline{Z}}}_\text{I} 
    +\underbrace{\bigang{-\jump\nabla_W \dot\alpha, \overline{Z} }}_\text{II}
    +\underbrace{\bigang{-\jump \dot\alpha, \pa_\delta T_i\overline{ D_tZ }}}_\text{III}
    +\underbrace{\bigang{-\jump \dot\alpha,\overline{\nabla_W Z }}}_\text{IV};
\end{aligned}
\end{equation}
here, for brevity, we have omitted the evaluation of each term at $(T_i(\ep,\delta),\ep,\delta)$.
Since $\alpha$ satisfies $D_t \dot \alpha=-\nabla V$ both before and after $T_i$, with $\nabla V$ continuous, the term \RN{1} is zero, and we focus on \RN{2}, \RN{3}, \RN{4}. 
We split into the cases where $T_i\in\Kt$ and $T_i\in\Rt$.

If $T_i\in\Kt$, then $\jump\dot\alpha = 0$, i.e.\ the terms \RN{3}, \RN{4} both vanish. Hence, for $T_i\in\Kt$, we get
\[\frac{\partial}{\partial\delta}\left(\bigang{-\jump\dot\alpha(T_i(\ep,\delta),\ep,\delta),\ol{Z}(T_i(\ep,\delta),\ep,\delta)}\right)\Big\rvert_{\ep=0,\delta=0} = \bigang{-\jump\nabla_W\dot\alpha,\ol{Z}}|_{T_i} = \bigang{-\jump D_tW(T_i),\ol{Z}(T_i)},\]
where we rewrite $\nabla_W\dot\alpha = D_tW$ using the vanishing of the torsion.

We now focus on $T_i\in\Rt$.	Fixing this $T_i$ for the moment, we
  will employ the more concise notation
  \begin{equation}\begin{aligned}
\alpha^+ &= \alpha\rvert_{t \in [T_i(\ep,\delta),
  T_{i+1}(\ep,\delta)]},\\
\alpha^- &= \alpha\rvert_{t \in [T_{i-1}(\ep,\delta), T_{i}(\ep,\delta)]}.
\end{aligned}
\end{equation}
for the successive smooth segments of $\alpha$.  We will further
abbreviate by writing simply $\dot\alpha^\pm$ for the evaluation of this
time derivative at time $T_i\pm$.
	Recalling that by definition of reflected \physs, we have
          $\jump \dot \alpha = -2
        \dot\alpha_\nrml^-,$  we rewrite the remaining terms as 
\begin{equation}\label{234}
  \begin{aligned}
\underbrace{\bigang{-\jump\nabla_W \dot\alpha, \overline{Z} }}_\text{II}
+ \underbrace{\bigang{2\dot\alpha_\nrml^-,\pa_\delta T_i \overline{D_tZ }}}_\text{III}
+ \underbrace{\bigang{2\dot\alpha_\nrml^-,\overline{\nabla_W Z }}}_\text{IV}.
\end{aligned}
\end{equation}
Let 
\begin{equation}
\label{c} 
c(\delta,\ep) := \alpha^\pm(T_i(\ep,\delta),\ep,\delta) \in Y\end{equation} 
(with choice of $\pm$ irrelevant) be the point at which $\alpha$ reflects.  Note that as a consequence, $\pa_\delta c \in TY$.
Recall that the second fundamental form is defined by $$\second(X,Y)=-\ang{Y, \nabla_X N}N = \ang{\nabla_X Y, N}N$$ (with the second equality obtained using compatibility of the connection and vanishing of its torsion).
For later use we also introduce the shape operator $\shape$, given by \begin{equation}\label{def.shape}\second(V,W) =
    \ang{\shape(V), W} N.\end{equation}

\begin{lemma}\label{lemma:normalsum}
The averaged normal components satisfy the following relation:
  \begin{equation}\label{normalsum}
\pa_\delta T_i\ol{D_tZ}_\nrml+\ol{\nabla_Z W}_\nrml  = \pa_\delta T_i \pa_\ep T_i (\nabla V)_\nrml - \pa_\ep T_i \ol{D_t W}_\nrml +\second(\pa_\delta c, \pa_\ep T_i \dot \alpha_\tgt + \ol {Z}).
  \end{equation}
  \end{lemma}
  \begin{proof}
    We recall that
    $$
\alpha^\pm(T_i(\ep,\delta),\ep, \delta) \in Y\quad \text{ for all } \ep, \delta.
    $$
    Differentiating in $\delta$ yields
\begin{equation}\label{normalcprime}
\ang{\dot \alpha^\pm \pa_\delta T_i + W^\pm, N}=0.
\end{equation}
Further differentiating in $\ep$ then gives (omitting a $\pm$ on all terms)
$$
\begin{aligned}
&\bigang{ D_t Z \pa_\delta T_i + D_t \dot\alpha \pa_\ep T_i \pa_\delta T_i +\dot \alpha \partial^2_{\ep\delta}T_i + \nabla_Z W + D_t W \pa_\ep T_i, N}\\
& + \ang{\pa_\delta c, \nabla_Z N+\pa_\ep T_i D_t N}=0.
\end{aligned}
$$
(Here we have used $\nabla_Z \dot \alpha= D_t Z$ by the vanishing of the torsion.)
We now take the average of this equation over $\pm$ and recall that $D_t\dot \alpha=-\nabla V$, while $\ang{\overline{\dot \alpha}, N}=0$ to obtain
$$
\bigang{ \ol{D_t Z} \pa_\delta T_i -\nabla V \pa_\ep T_i \pa_\delta T_i + \ol{\nabla_Z W} + \ol{D_t W} \pa_\ep T_i, N}
+ \ang{\pa_\delta c, \nabla_{\ol{Z}}
  N+\pa_\ep T_i D_t N}=0.
$$
Now recalling the definition of the second fundamental form we rewrite our identity as
$$
\bigang{ \ol{D_t Z} \pa_\delta T_i -\nabla V \pa_\ep T_i \pa_\delta T_i + \ol{\nabla_Z W} + \ol{D_t W} \pa_\ep T_i, N}
 - \ang{\second(\pa_\delta c, \ol{Z}+\pa_\ep T_i \dot \alpha_\tgt),N}=0,
$$
as desired.    
\end{proof}

Since $\nabla _Z W=\nabla_W Z$ we may now substitute \eqref{normalsum} into \eqref{234}, using it to replace the terms \RN{3} and \RN{4} with terms involving the LHS of \eqref{normalsum} and get
$$
  \bigang{-\jump D_t W, \overline{Z} }
 +\bigang{2\dot\alpha_\nrml^-,\pa_\delta T_i \pa_\ep T_i \nabla V -\pa_\ep T_i \ol{D_t W} + \second(\pa_\delta c, \ol{Z}+\pa_\ep T_i \dot \alpha_\tgt)}.
$$
Since $\alpha^\pm (T(\ep,\delta), \ep, \delta)=0$,
\begin{equation}\label{constraint}
\ang{\dot\alpha^\pm \pa_\delta T_i +W^\pm, N}=\ang{\dot\alpha^\pm \pa_\ep T_i +Z^\pm, N}=0.
\end{equation}
Substituting in the above yields for boundary terms in the second variation.
$$
  \bigang{-\jump D_t W, \overline{Z} }
 +\bigang{2W_\nrml^-, (Z_\nrml^-/\dot \alpha_\nrml^-) \nabla V} +2\ang{Z^-_\nrml ,\ol{D_t W} }+ \ang{2 \dot\alpha_\nrml^-,\second(\pa_\delta c, \ol{Z}+\pa_\ep T_i \dot \alpha_\tgt)}.
$$
Now use \eqref{constraint} to eliminate $\pa_\ep T_i$ in favor of the
variation vector field $Z$ (and recall the definition
\eqref{def.shape} of the shape operator), to find that this sum equals
$$
  \bigang{-\jump D_t W, \overline{Z} } +2\ang{\ol{D_t W}, Z^-_\nrml }
  +\bigang{2W_\nrml^-, (Z_\nrml^-/\dot \alpha_\nrml^-) \nabla V}
+ 2 \dot \alpha_1 \ang{\shape(\pa_\delta c), \ol{Z}}
- 2\second(\pa_\delta c, \dot \alpha_\tgt)Z_1^-.
$$

Summarizing the above discussion, we obtain:
\begin{theorem}\label{theorem:second}
Let $\alpha(t,\ep,\delta)\in\Omega_0(M;p,p')$ be a two parameter variation of a
\reflective \phys $\alpha(t)$ with variational vector field $\pa_\ep
\alpha(t,0,0)=Z, \pa_\delta \alpha(t,0,0)=W\in T_{\alpha}\Omega_0(M,p,p')$.  Then the second variation $\frac{\partial^2}{\partial\epsilon \partial\delta} 
     (J[\alpha(\cdot;\ep,\delta)])\big\rvert_{\substack{\epsilon=0, \delta=0}} $ is given by 
\begin{equation}
\label{eq:twovar-summary}
    \begin{aligned}
        &-\int_0^T\bigang{D^2_t W + R(\dot\alpha, W) \dot \alpha+
          (\nabla^2 V) W, Z} \,dt 
           + \sum_{T_i\in\Rt} \bigg( -  \bigang{\jump D_t W, \overline{Z} } +2\ang{\ol{D_t W}, Z^-_\nrml } \\
& +\bigang{2W_\nrml^-, (Z_\nrml^-/\dot \alpha_\nrml^-) \nabla V}
	+ 2 \dot \alpha_1 \ang{\shape(\pa_\delta c), \ol{Z}}
- 2\second(\pa_\delta c, \dot \alpha_\tgt)Z_1^-\bigg)_{T_i}
 - \sum_{T_i\in\Kt} \bigang{\jump D_tW (T_i),\ol{Z}(T_i) }
    \end{aligned}
\end{equation}
where
\begin{equation}
    \pa_\delta c = -\frac{W_\nrml}{\dot\alpha_\nrml}\dot \alpha_\tgt + W_\tgt.
\end{equation}
\end{theorem}
As written here, the second variation is manifestly linear in $Z$.  It
is also symmetric in $Z$ and $W$ owing to its definition as a second
derivative. As every element in $T_{\alpha}\Omega_0(M,p,p')$ arises
from differentiating a variation (cf.~Lemma \ref{lemma:integration}),
we obtain a quadratic form on the tangent space:
\begin{definition}
  \label{def:form}
For $\alpha\in \Omega_0(M,p,p')$, we define a symmetric bilinear quadratic form
\begin{equation}
J''(\cdot, \cdot): T_{\alpha}\Omega_0(M,p,p')\times  T_{\alpha}\Omega_0(M,p,p') \longrightarrow \RR
\end{equation}
using equation \eqref{eq:twovar-summary}.
\end{definition}

\subsection{Reflected Jacobi fields for reflected physical trajectories}

The result of Theorem~\ref{theorem:second} now motivates the following
definition (the essential point being
Proposition~\ref{proposition:null} below). 
\begin{definition}
  \label{def:Jacobi}
    A vector field $W\in T_{\alpha}\Omega(M)$ along a \reflective
    \phys $\alpha(t)$ with reflection and kink times at $\Rt$ and $\Kt$, respectively, is called a \emph{\reflective Jacobi field} if it satisfies the Jacobi equation
\begin{equation}
\label{eq:Jacobi}
    {D_t^2 W} + R(\dot\alpha, W) \dot \alpha+ (\nabla^2 V) W = 0
\end{equation}
on $[0,T]\backslash(\Rt\cup\Kt)$, as well as the reflection conditions
\begin{gather}
    \label{eq:Jacobi_b1}
    \jump D_t W_\tgt=2 \dot \alpha_1^- \shape(\pa_\delta c),\\
    \label{eq:Jacobi_b2}
    \ol{D_t W}_\nrml=- (W_\nrml^-/\dot\alpha_\nrml^-)  (\nabla V)_\nrml +  \second(\pa_\delta c,\dot\alpha_\tgt)
\end{gather}
at $t=T_i\in\Rt$, where
\begin{equation}
    \pa_\delta c = -\frac{W_\nrml}{\dot\alpha_\nrml}\dot \alpha_\tgt + W_\tgt,
\end{equation}
and the kink conditions
\begin{equation}
\label{eq:Jacobi_kink}
\jump D_tW = 0
\end{equation}
at $t=T_i\in\Kt$.
\end{definition}
Note that a \reflective Jacobi field is determined completely by its
initial conditions $W(0)$, $\cov W (0)$ (and depends smoothly on
  them). Indeed, for $T_i\in\Rt\cup\Kt$, the values of $W(T_i-)$ and $\cov W(T_i-)$ uniquely determine the values of $W(T_i+)$ (via the requirement $W\in T_\alpha\Omega(M)$) and $\cov W(T_i+)$ (via the reflection/kink conditions above).	
	As with \reflective \physs,
  passing over $Y$ at points where $\dot \alpha$ is transverse to $Y$
  creates no difficulties with solvability or with smooth dependence
  on initial data.  At internal kinks, note that  Jacobi fields must be smooth.
\begin{remark}
Note that, in the definition, we do not require that a reflected Jacobi field vanish at endpoints. We allow the case in which no reflections occur, in which case the definition coincides with the usual definition of Jacobi fields.
\end{remark}
Let $p=\alpha(a)$ and $q=\alpha(b)$ ($a\neq b$) be two points on a \reflective \phys $\alpha(t)$. In particular, we are allowing $p,q\in Y$ or $p=q$.
\begin{definition}
\label{def:conjugate}
    The points $p$ and $q$ are \emph{conjugate} along $\alpha(t)$ if there exists a non-vanishing \reflective Jacobi field $W$ along $\alpha(t)$ such that $W(a)=W(b)=0$. The \emph{multiplicity} of $p$ and $q$ as conjugate points is equal to the dimension of the vector space consisting of all such \reflective Jacobi fields.
\end{definition}
Recall that the \emph{null space} of the second variation
$J'': T_{\alpha}\Omega_0\times T_{\alpha}\Omega_0 \to \RR$
is the vector space consisting of those $W\in T_{\alpha}\Omega_0$ such that $J''(W,Z)=0$ for
all \reflective variation vector field $Z\in T_{\alpha}\Omega_0$. The \emph{nullity} $\nu$ of $J''$ is equal to the dimension of this null space. We say $J''$ is \emph{degenerate} if $\nu>0$. We have the following Proposition
\begin{proposition}\label{proposition:null}
    A vector field $W\in T_{\alpha}\Omega_0$ belongs to the null space of $J''$ if and only if $W$ is a \reflective Jacobi field. Therefore $J''$ is degenerate if and only if the end points $p$ and $q$ are conjugate along $\alpha(t)$. The nullity of $J''$ is equal to the multiplicity of $p$ and $q$ as conjugate points.
\end{proposition}
\begin{proof}
  If $W$ is a \reflective Jacobi field and vanishes at $p$ and $q$, comparing equation \eqref{eq:twovar-summary} with Definition \ref{def:path.global}, it is easy to see 
that $J''(W,Z)$ vanishes for all $Z\in T_{\alpha}\Omega_0$.

  Assume now $W\in T_{\alpha}\Omega_0$ belongs to the null space of $J''$. 
	Let $Z_1$ be any smooth vector field vanishing at all $T_i\in\Rt\cup\Kt$; note that $Z_1\in T_\alpha\Omega_0$ by Corollary \ref{cor:physz}. Then in computing $J''(W,Z_1)$ we see that all boundary terms vanish, leaving
	\[0 = J''(W,Z_1) = -\int_0^T\bigang{D^2_t W + R(\dot\alpha, W) \dot \alpha+
          (\nabla^2 V) W, Z_1} \,dt.\]
  Applying this to arbitrary $Z_1$ vanishing on $\Rt\cup\Kt$, it follows that $W$ satisfies $D^2_tW + R(\dot\alpha,W)\dot\alpha + (\nabla^2V)W = 0$ on $[0,T]\backslash(\Rt\cup\Kt)$, i.e.\ $W$ satisfies condition \eqref{eq:Jacobi}.
  
  Next, we take a smooth $Z_2$ which vanishes on $\Rt$, but not necessarily on $\Kt$. For such $Z_2$, we have
	\[0 = J''(W,Z_2) = - \sum_{T_i\in\Kt} \bigang{\jump D_tW(T_i),\ol{Z_2}(T_i)}\]
	since all boundary terms at $T_i\in\Rt$ vanish, and the integral term vanishes since we already have \eqref{eq:Jacobi}. By Corollary \ref{cor:physz}, we can take $Z_2$ to take on arbitrary values at $T_i\in\Kt$. Thus we obtain $\Delta D_tW(T_i) = 0$ for all $T_i\in\Kt$.
	Finally, we consider $Z_3$ which do not vanish at $T_i\in\Rt$.
        Specifically we take $Z_3$ to vanish at $t=0$ and $t=T$ and satisfy
	\[Z_{3,\tgt}(T_i\pm) = \jump D_t W_\tgt(T_i) - 2 \dot \alpha_1^-(T_i) \shape(\pa_\delta c)(T_i)\]
	and 
	\[Z_{3,\nrml}(T_i\pm) = \pm\left(\ol{D_t W}_\nrml(T_i) + (W_\nrml^-(T_i)/\dot\alpha_\nrml^-(T_i))  (\nabla V)_\nrml(\alpha(T_i)) +  \second(\pa_\delta c,\dot\alpha_\tgt)(T_i)\right)\]
	when $T_i\in\Rt$.	Note that $Z_3\in T_\alpha\Omega_0(M)$ by Corollary \ref{cor:physz}, since $\jump Z_3(T_i)$, resp. $\ol{Z_3}(T_i)$, is normal, resp. tangent, to $Y$ at $\alpha(T_i)$. Then
  \begin{equation*}
    \begin{split}
      0 = J''(W,Z_3) &= - \sum_{i=1}^m \left(\|\jump D_t W_\tgt - 2 \dot \alpha_1^- \shape(\pa_\delta c)\rvert_{T_i}\|^2_g\right. \\
       & \left.+\|\ol{D_t W}_\nrml + (W_\nrml^-/\dot\alpha_\nrml^-)  (\nabla V)_\nrml +  \second(\pa_\delta c,\dot\alpha_\tgt)\rvert_{T_i}\|^2_g\right),
    \end{split}
  \end{equation*}
  which yields condition \eqref{eq:Jacobi_b1} and \eqref{eq:Jacobi_b2}. Therefore, $W$ must be \reflective Jacobi field by Definition \ref{def:Jacobi}.
\end{proof}

\subsection{Reflected Jacobi fields as variation of \physs}
Let $\alpha(t,\ep)\in \Omega(M)$ be a 1-parameter variation of $\alpha(t)$ in the sense that $\alpha(t,0) = \alpha(t)$, not necessarily keeping the endpoints fixed, and such that each $\alpha(\cdot, \ep)$ is a \reflective \phys for any fixed $\ep\in(-\ep_0,\ep_0)$. In fact, such variation is given by a family of \reflective \phys.
\begin{proposition}\label{prop:physicalvar}
    If $\alpha(t,\ep)\in\Omega(M)$ is variation of a \reflective \phys
    as above, then the corresponding variation vector field $W(t) = \frac{\pa \alpha}{\pa \ep}(t,0)\in T_{\alpha}\Omega(M)$ is a \reflective Jacobi field along $\alpha(t)$.
  \end{proposition}
	\begin{remark}
 Note that this result (and hence the reflection
  conditions for Jacobi fields) was previously known in the Riemannian geometry case
  $V=0$---see \cite[Theorem 3.13]{KaZh:03}, \cite[Equation (2)]{Wo:07}, \cite[Lemma 16]{ZhLi:07}, \cite[Lemma 12]{IlSa:16}, etc.
\end{remark}
\begin{proof}
The idea to derive the conditions for \reflective Jacobi fields is to take the conditions for a reflected \phys (cf.~Definition \ref{def:path.global}) and differentiate. 

For the condition \eqref{eq:Jacobi} on the interior points, note that as \reflective \physs in the interior point satisfying
$\cov \dot \alpha + \nabla V(\alpha(t,\delta)) = 0,$
we have
\begin{equation}
    \begin{split}
        0 = D_{\delta}(\cov \dot \alpha + \nabla V(\alpha(t,\delta))) = D_t D_{\delta} \dot \alpha+ R\big( \dot\alpha, \frac{\pa \alpha}{\pa \delta}\big) \dot \alpha + \nabla^2 V(\alpha(t,\delta)) \frac{\pa \alpha}{\pa \delta},
    \end{split}
\end{equation}
where the RHS of the above equation is indeed \eqref{eq:Jacobi} if we evaluate at $\delta=0$.

For the conditions \eqref{eq:Jacobi_b1}\eqref{eq:Jacobi_b2} at the reflection times $T_i\in\Rt$, write $\alpha^{\pm}(t,\delta)$ as the restriction of $\alpha(t,\delta)$ to $t>T_i(\delta)$ or $t<T_i(\delta)$, extended smoothly to a neighborhood of $T_i(\delta)$. Note that 
we have the reflection conditions
\[\dot\alpha_\nrml^-(T_i(\delta),\delta) + \dot\alpha_\nrml^+(T_i(\delta),\delta) = 0, \quad
\dot\alpha_\tgt^-(T_i(\delta),\delta) - \dot\alpha_\tgt^+(T_i(\delta),\delta) = 0,\]
since $\alpha(\cdot,\delta)$ is a reflected \phys for each $\delta$.
Differentiating in $\delta$ gives
\begin{gather*}
  \left(\pa_\delta T_i D^-_t +  D^-_{\delta}\right)\dot\alpha_\nrml^- + \left(\pa_\delta T_i D^+_t +  D^+_{\delta}\right)\dot\alpha_\nrml^+ = 0, \\
  \left(\pa_\delta T_i D^-_t +  D^-_{\delta}\right)\dot\alpha_\tgt^- - \left(\pa_\delta T_i D^+_t +  D^+_{\delta}\right)\dot\alpha_\tgt^+ = 0,
\end{gather*}
where $D^{\pm}$ denotes taking the covariant derivative along $\alpha^{\pm}$, and all terms are evaluated at $t=T_i(0)$, $\delta=0$. Recalling the notation $c(\delta) = \alpha^{\pm}(T_i(\delta),\delta)$ (in which case $c(\delta)\in Y$ and $\partial_\delta c\in TY$, with $\partial_\delta c = \pa_\delta T_i\dot\alpha^{\pm} + \frac{\partial\alpha^{\pm}}{\partial\delta}$), we have
\begin{equation}
\label{eq:covderivnrml}
\begin{aligned}
\left(\pa_\delta T_i D^{\pm}_t +  D^{\pm}_{\delta}\right)\dot\alpha_\nrml^{\pm} &= \left(\pa_\delta T_i D^{\pm}_t +  D^{\pm}_{\delta}\right)\left(\langle\dot\alpha^{\pm},N\rangle N\right) \\
&= \left(\left\langle\left(\pa_\delta T_i D^{\pm}_t +  D^{\pm}_{\delta}\right)\dot\alpha^{\pm},N\right\rangle + \left\langle\dot\alpha^{\pm},\left(\pa_\delta T_i D^{\pm}_t +  D^{\pm}_{\delta}\right)N\right\rangle\right) N \\
&+ \langle\dot\alpha^{\pm},N\rangle\left(\pa_\delta T_i D^{\pm}_t +  D^{\pm}_{\delta}\right)N \\
&= \left(\left\langle\left(\pa_\delta T_i D^{\pm}_t +  D^{\pm}_{\delta}\right)\dot\alpha^{\pm},N\right\rangle + \left\langle\dot\alpha^{\pm},\nabla_{\partial_\delta c}N\right\rangle\right)N + \dot\alpha_1^{\pm}\nabla_{\partial_\delta c}N \\
&= \left\langle\left(\pa_\delta T_i D^{\pm}_t +  D^{\pm}_{\delta}\right)\dot\alpha^{\pm},N\right\rangle N - \second(\dot\alpha_{\tgt}^{\pm},\partial_\delta c) - \dot\alpha_1^{\pm}\shape(\partial_\delta c).
\end{aligned}
\end{equation}
Since
\[ D^{\pm}_t\dot\alpha^{\pm} + \nabla V(\alpha^{\pm}) = 0,\]
it follows that
\begin{equation}
\label{eq:covderivalpha}
\left(\pa_\delta T_i D^{\pm}_t +  D^{\pm}_{\delta}\right)\dot\alpha^{\pm} =  D^{\pm}_{\delta}\dot\alpha^{\pm} - \pa_\delta T_i\nabla V = \dot W^{\pm}-\pa_\delta T_i\nabla V.
\end{equation}
Thus, combining \eqref{eq:covderivnrml} and \eqref{eq:covderivalpha} yields
\begin{align*}
\left(\pa_\delta T_i D^{\pm}_t +  D^{\pm}_{\delta}\right)\dot\alpha_{\nrml}^{\pm} &= \left\langle\dot W^{\pm}-\pa_\delta T_i\nabla V,N\right\rangle N - \second(\dot\alpha_{\tgt}^{\pm},\partial_\delta c) - \dot\alpha_1^{\pm}\shape(\partial_\delta c)\\
&=\dot W_{\nrml}^{\pm} - \pa_\delta T_i(\nabla V)_{\perp} - \second(\dot\alpha_{\tgt}^{\pm},\partial_\delta c) - \dot\alpha_1^{\pm}\shape(\partial_\delta c).
\end{align*}
Averaging the above equation over $\pm$, and using that $\sum_{\pm}\dot\alpha_{\nrml}^{\pm} = 0$, yields
\[0 =\sum_{\pm}\left(\pa_\delta T_i D^{\pm}_t +  D^{\pm}_{\delta}\right)\dot\alpha_{\nrml}^{\pm} \implies  \overline{\dot W_{\nrml}} = \pa_\delta T_i(\nabla V)_{\nrml} + \second(\dot\alpha_{\tgt},\partial_\delta c) = -\frac{W_{\perp}^-}{\dot\alpha_{\perp}^-}(\nabla V)_{\nrml} + \second(\dot\alpha_{\tgt},\partial_\delta c),\]
thus giving \eqref{eq:Jacobi_b1}. Moreover,
\begin{align*}
\left(\pa_\delta T_i D^{\pm}_t +  D^{\pm}_{\delta}\right)\dot\alpha_\tgt^{\pm} &= \left(\pa_\delta T_i D^{\pm}_t +  D^{\pm}_{\delta}\right)\dot\alpha^{\pm} - \left(\pa_\delta T_i D^{\pm}_t +  D^{\pm}_{\delta}\right)\dot\alpha_\nrml^{\pm} \\
&= \dot W^{\pm} - \pa_\delta T_i\nabla V - \left(\dot W_{\nrml}^{\pm} - \pa_\delta T_i(\nabla V)_{\perp} - \second(\dot\alpha_{\tgt}^{\pm},\partial_\delta c) - \dot\alpha_1^{\pm}\shape(\partial_\delta c)\right)
\end{align*}
so we have
\[0 = \Delta\left(\left(\pa_\delta T_i D^{\pm}_t +  D^{\pm}_{\delta}\right)\dot\alpha_\tgt^{\pm}\right) = \Delta \dot W - \Delta \dot W_{\nrml} + \Delta\dot\alpha_1^{\pm}\shape(\partial_\delta c)\]
since $\pa_\delta T_i\nabla V$ and $\dot\alpha_{\tgt}$ are the same across the jump. Note that $\Delta\dot W - \Delta\dot W_\nrml = \Delta\dot W_\tgt$. Hence we have
\[\Delta(\dot W_\tgt) = -\Delta\dot\alpha_1^{\pm}\shape(\partial_\delta c) = 2\dot\alpha_1^-\shape(\partial_\delta c),\]
thus giving \eqref{eq:Jacobi_b2}.

Finally, \eqref{eq:Jacobi_kink} follows similarly by differentiating the condition $\jump\dot\alpha(T_i(\delta),\delta) = 0$ in $\delta$ whenever $T_i\in\Kt$.
\end{proof}

\begin{proposition}
    Every \reflective Jacobi field along a \reflective \phys $\alpha: [0,T]\to M$ may be obtained by a variation through \reflective \physs.
\end{proposition}
\begin{proof}
Let $W$ be a \reflective Jacobi field along $\alpha$.  Let
$\alpha(t,\ep)$ be a family of \reflective \physs with $\alpha(t,0)=\alpha(t)$,
$\pa_\ep \alpha(0, 0) = W(0)$, $D_t \pa_\ep \alpha(0,0)=\dot W(0)$.
Then $Z:= \pa \alpha/\pa \ep\rvert_{\ep=0}$ is
a \reflective Jacobi field satisfying the same initial conditions as
$W$, hence $W=Z$.
\end{proof}
We record the following proposition which will be useful later. 
\begin{proposition}\label{prop:twopoint}
  If $p,p'\in M$ are non-conjugate along a \reflective \phys
  $\alpha(t)$ with $\alpha(0)=p, \alpha(T)=p'$, then for any pair of
  vectors $(V_0, V_T)\in T_pM\times T_{p'}M$, there exists a unique
  \reflective Jacobi field $W(t)$ along $\alpha(t)$ such that $W(0) =
  V_0$ and $W(T) = V_T$. 
\end{proposition}
\begin{proof}
Recall that a reflected Jacobi field along $\alpha$ exists and is
unique given its initial data.  Now given $V, Z \in T_{\alpha(0)} M$, let
$W_{V,Z}(t)$ denote the Jacobi field with $$W_{V,Z}(0)=V,\quad D_t
W_{V,Z}(0)=Z;$$ note that this depends bilinearly on $(V,Z)$.  Consider the map
$\Phi: T_{\alpha(0)}M\to T_{\alpha(T)}M$ defined by
$$
\Phi(Z):=W_{0, Z}(T).
$$
This linear map is injective since $\alpha(0)$ and $\alpha(T)$ are
non-conjugate, hence is an isomorphism.  Now given $V_0$ and $V_T$,
let $$Z_0=\Phi^{-1}(V_T-W_{V_0,0}(T)).$$
Then $W_{V_0, Z_0}(T)=V_T$ as desired.
\end{proof}
By Lemma \ref{lem:existuniq}, given a \reflective \phys $\alpha(t)$,
there is a neighborhood of $(\alpha(0),\dot\alpha(0))\in TM$
such that for $(x,v)$ in this neighborhood,
there exists a unique \reflective \phys, with initial location and velocity $(x,v)$,
and with reflection times close to that of $\alpha$.
Let $\alpha_{x,v}(t)$ denote this path.
\begin{proposition}\label{prop:locdiffeo}
  If $p,p'\in M$ are non-conjugate along a \reflective \phys
  $\alpha(t)$ with $\alpha(0)=p,\ \alpha(T)=p'$,  then the map
  $$
\Psi=(x,v) \mapsto (x, \alpha_{x,v}(T))
$$
is a local diffeomorphism from a neighborhood of $(p, \dot\alpha(0))$
in $TM$ to a neighborhood of $(p,p')$ in $M\times M$.
  \end{proposition}
\begin{proof}
The derivative of $\Psi$ is invertible by
Proposition~\ref{prop:twopoint}, hence the result follows from the
Inverse Function Theorem.
  \end{proof}

Given $x$ and $y$ in the range of the local diffeomorphism defined by
Proposition~\ref{prop:locdiffeo}, we let $\alpha_{x,y}(t)$ denote the
resulting reflected physical path from $x$ to $y$.

\section{The index theorem}\label{sec:index}
We are ready to prove the Morse index theorems in this context. These are theorems which relate the \emph{index} of $J''$, on certain path spaces, with geometric quantities such as the number of conjugate points along the path.  We will prove such theorems in the context of fixed boundary conditions, concatenation of two paths, and periodic boundary condition; additional setup is needed in the last case.

We recall that given a quadratic form $Q$ on a (possibly infinite-dimensional) vector space $V$, its \emph{index} $\ind(Q)$ is the maximum dimension of a subspace on which $Q$ is negative definite. An important fact which we will frequently use in this section is the following lemma\footnote{If $V$ is finite
dimensional, this lemma follows by diagonalizing $Q$ on $V_1$ and $V_2$ with
appropriate choices of inner product on $V$ and bases on $V_1, V_2$;
this is is a special case of the \emph{Haynsworth inertia additivity formula}; see \cite{PuSt:05} for a reference. We were unable to find a proof in the literature in the case that $V$ is not finite-dimensional; however, the result follows from the finite-dimensional case, as follows: if $W$ is a finite-dimensional subspace of $V$ on which $Q$ is negative definite, and $W_1$, $W_2$ are subspaces of $V_1$, $V_2$ on which $Q$ is negative definite with maximal dimension, consider
\[\tilde{V} = W+W_1+W_2,\quad \tilde{V}_1 = \pi_1(W)+W_1,\quad \tilde{V}_2 = \pi_2(W)+W_2\]
where $\pi_1,\pi_2$ are the projections from $V$ onto $V_1$,
$V_2$. Then
$\ind{Q|_{\tilde{V}}} =
\ind{Q|_{\tilde{V}_1}}+\ind{Q|_{\tilde{V}_2}}$ since $\tilde V$
  is finite dimensional. Moreover
\[\dim(W)\le\ind{Q|_{\tilde{V}}} = \ind(Q|_{\tilde{V}_1})+\ind(Q|_{\tilde{V}_2}) = \ind(Q|_{V_1}) + \ind(Q|_{V_2}),\]
where the last equality follows since $\tilde{V}_1$, $\tilde{V}_2$ already contain a maximal-dimensional subspace of $V_1$, $V_2$ where $Q$ is negative definite. This shows $\ind(Q)$ is finite and is at most $\ind(Q|_{V_1})+\ind(Q|_{V_2})$; the other inequality $\ind(Q)\ge\ind(Q|_{V_1})+\ind(Q|_{V_2})$ is obvious.}:
\begin{lemma}
\label{lem:ind-add}
If $Q$ is a quadratic form on a vector space $V$, and $V = V_1\oplus V_2$, where the direct sum is orthogonal with respect to $Q$, then
\[\ind(Q) = \ind(Q|_{V_1}) + \ind(Q|_{V_2}),\]
assuming $\ind(Q|_{V_1}), \ind(Q|_{V_2})<\infty$.
\end{lemma}

\subsection{The Morse index theorem for fixed endpoints}
We follow Milnor's treatment \cite{Mi:63} with minor modifications. The main theorem is:
\begin{theorem}[Morse Index Theorem with fixed endpoints]
\label{thm:morse1}
\[\ind(J''|_{T_\alpha\Omega_0(M;p,p')}) = \text{number of conjugate
    points along }\alpha \text{ with respect to $p$}\]
\end{theorem}
Before proving the theorem, we begin with a lemma insuring that sufficiently short (reflected)
paths are locally action-minimizing.

\begin{lemma}\label{lemma:locmin}
Fix a \reflective \phys $\alpha(t)$, $t \in [0,T]$. For $\ep>0$, let
\begin{align*}
V_\ep =& \left\{Z\in T_\alpha\Omega_0(M;p,p')\,:\,\text{for all }t\in[0,T]\text{, there exists }t'\in[0,T] \right.\\
&\left.\text{with }|t'-t|<\ep\text{ such that }Z(t')=0\right\}.
\end{align*}
For $\ep>0$ sufficiently small, $J''$ is positive definite on $V_\ep$.
\end{lemma}

\begin{proof}
By equation \eqref{eq:twovar-summary}, we have $J''(Z,Z) = I(Z,Z)+B(Z,Z)$ with
\[I(Z,Z) = -\int_0^T\bigang{{D_t^2 Z} + R(\dot\alpha, Z) \dot \alpha+ (\nabla^2 V) Z, Z} \,dt,\]
and $B(Z,Z)$ the ``boundary terms'' aside from the integral. Integrating by parts on each interval of the form $[T_i,T_{i+1}]$ yields
\[
I(Z,Z) = \int_0^T \bigang{\cov Z,\cov Z} - \bigang{R(\dot\alpha, Z) \dot \alpha+ (\nabla^2 V) Z, Z} \, dt  +\sum_{i=1}^n \jump\left(\bigang{\cov Z,Z}\right)(T_i);
\]
note that we can also write
\[\jump\left(\bigang{\cov Z,Z}\right)(T_i) = \bigang{\jump\cov Z(T_i),\overline{Z}(T_i)} + \bigang{\overline{\cov Z}(T_i),\jump Z(T_i)}.\]
For $T_i\in\Kt$, we have $\jump Z(T_i) = 0$ since $\alpha$ is a reflected \phys. Hence,
\[\sum_{i=1}^n\jump\left(\bigang{\cov Z,Z}\right)(T_i) = \sum_{T_i\in\Rt}\left(\bigang{\jump\cov Z,\overline{Z}} + \bigang{\overline{\cov Z},\jump Z}\right)\Big\rvert_{T_i} + \sum_{T_i\in\Kt}\bigang{\jump\cov Z,\overline{Z}}\Big\rvert_{T_i}.\]
It follows that
\[J''(Z,Z) = \int_0^T\bigang{\cov Z,\cov Z} - \bigang{R(\dot\alpha, Z) \dot \alpha+ (\nabla^2 V) Z, Z} \, dt  + \tilde{B}(Z,Z)\]
where
\begin{align*}
\tilde{B}(Z,Z) &= B(Z,Z) + \sum_{T_i\in\Rt}\left(\bigang{\jump\cov Z,\overline{Z}} + \bigang{\overline{\cov Z},\jump Z}\right)\Big\rvert_{T_i} + \sum_{T_i\in\Kt}\bigang{\jump\cov Z,\overline{Z}}\Big\rvert_{T_i} \\
&= \sum_{T_i\in\Rt}\left(      2\ang{Z^-_\nrml ,\ol{D_t Z} } + \bigang{\overline{D_tZ},\jump Z}
    +\bigang{2Z_\nrml^-, (Z_\nrml^-/\dot \alpha_\nrml^-) \nabla V}  + \ang{2 \dot\alpha_\nrml^-,\second(\pa_\ep c, \ol{Z}+T'_\ep \dot \alpha_\tgt)}\right)\Big\rvert_{T_i} \\
&=\sum_{T_i\in\Rt}\left(\bigang{2Z_\nrml^-, (Z_\nrml^-/\dot \alpha_\nrml^-) \nabla V}  + \ang{2 \dot\alpha_\nrml^-,\second(\pa_\ep c, \ol{Z}+T'_\ep \dot \alpha_\tgt)}\right)\Big\rvert_{T_i}
\end{align*}
since $\jump Z = -2Z^-_\nrml$.
Recalling that $\pa_\ep T_i$ and $\pa_\ep c$ are determined by the values of $Z$ at the reflection points via
\[\jump Z = -\pa_\ep T_i\jump\alpha,\quad \pa_\ep c = Z_\tgt + \pa_\ep T_i\alpha_\tgt,\]
it follow there exists a constant $C$ independent of $Z$ such that
\[|\tilde{B}(Z,Z)|\le C\max\limits_{[0,T]} \langle Z,Z\rangle.\]
Since there also exists $C'$ such that
\[|\bigang{R(\dot\alpha, Z) \dot \alpha+ (\nabla^2 V) Z, Z}|\le C'\langle Z,Z\rangle,\]
it follows that
\[J''(Z)\ge \|D_tZ\|_{L^2([0,T])}^2 - (C+C'T)\|Z\|_{L^\infty([0,T])}^2\]
where $\|W\|_{L^p([0,T])} := \||W(t)|_g\|_{L^p([0,T])}$. The claim is
that this quantity is non-negative (and strictly positive if
$Z\not\equiv 0$) if $\ep$ is sufficiently small. To verify this, we first take $\ep<\min\limits_{i=0,\dots,m}(T_{i+1}-T_i)$, i.e.\ $\ep$ to be smaller than the width of any subinterval on which $Z$ is smooth, in which case every $t\in[0,T]$ satisfies the property that there exists $t'\le t$, with $t-t'<\ep$, such that $Z(t')=0$ and $Z|_{(t',t)}$ is smooth. Then the Fundamental Theorem of Calculus holds on $[t',t]$, and we can write
\[\langle Z(t),Z(t)\rangle = \langle Z(t'),Z(t')\rangle + \int_{t'}^t 2\langle D_tZ(s),Z(s)\rangle\,ds = \int_{t'}^t 2\langle D_tZ(s),Z(s)\rangle\,ds.\]
It then follows that
\begin{align*}
\left|\langle Z(t),Z(t)\rangle\right| = \left|\int_{t'}^t 2\langle D_tZ(s),Z(s)\rangle\,ds\right| &\le 2\|D_tZ\|_{L^2([0,T])}\|Z\|_{L^2([t',t])} \\
&\le 2\ep^{1/2}\|D_tZ\|_{L^2([0,T])}\|Z\|_{L^\infty([0,T])}\\
&\le 2\ep\|D_tZ\|_{L^2([0,T])}^2 + \frac{1}{2}\|Z\|_{L^\infty([0,T])}^2,
\end{align*}
from which taking supremums and absorbing $\frac{1}{2}\|Z\|_{L^\infty([0,T])}^2$ into the LHS yields
\[\|Z\|_{L^\infty([0,T])}^2\le 4\ep\|D_tZ\|_{L^2([0,T])}^2.\]
Thus,
\[J''(Z)\ge (1-4(C+C'T)\ep)\|D_tZ\|_{L^2([0,T])}^2,\]
with $1-4(C+C'T)\ep>0$ for $\ep$ sufficiently small. Moreover, for such small $\ep$, we have $J''(Z)=0$ only if $D_tZ\equiv 0$, which when combined with $Z$ equaling $0$ at some times would happen only when $Z\equiv 0$.
\end{proof}
We are now ready to prove Theorem \ref{thm:morse1}:
\begin{proof}[Proof of Theorem \ref{thm:morse1}]
    We follow the classic treatment of \cite[\S 15]{Mi:63} in the
    standard setting, mainly noting where our setting of reflected
    trajectories (and mechanical rather than geometric Lagrangian
    function) requires changes.

    Fix times $0=t_0<t_1<\dots < t_k<T$ such that (for simplicity)
    $\alpha (t_j) \in M\backslash Y$ for all $j$ and sufficiently
    closely spaced that each
    pair $t_j, t_{j+1}$ can play the role of $t,t'$ in
      Lemma~\ref{lemma:locmin} above.

      Now let \begin{equation}\label{piecewiseJacobi} T_\alpha\Omega(t_0, \dots, t_k) \subset
      T_\alpha\Omega_0(M;p,p')\end{equation} denote the subspace of 
			``piecewise \reflective Jacobi fields,'' i.e.\ vector
      fields $W$ along $\alpha$ such that on each interval $[t_{j-1},
      t_j]$, $W$ is a \reflective Jacobi field and such that
      $W(0)=W(T)=0$.  Let $T'$ denote the space of $W \in
      T_\alpha\Omega_0(M;p,p')$ vanishing at $t_j$ for all $j=0,\dots,
      k$.
      
By the same reasoning as in \cite{Mi:63}, we now find
that $$T_\alpha\Omega_0(M;p,p')=T_\alpha\Omega(t_0, \dots, t_k) \oplus
T',$$ that the sum is orthogonal with respect to the quadratic form
$J''$, and that on $T'$ the form $J''$ is positive definite.  (The first
assertion follows from Theorem~\ref{theorem:second};  the latter assertion is where Lemma~\ref{lemma:locmin} is essential.)

Thus, the index of $J''$ is the same as the index of its restriction
to $T_\alpha\Omega(t_0, \dots, t_k)$.  As in \cite{Mi:63}, we now set
$\lambda(\tau)$ to be the value of the index $J''$ at $\alpha$
restricted to $t \in [0, \tau]$; this is nondecreasing and zero for sufficiently
small $\tau$ (using Lemma~\ref{lemma:locmin}) by the same reasoning on employed in \cite{Mi:63}.

Note that we may identify $T_\alpha\Omega(t_0, \dots,
t_k)$ with $T_{\alpha(t_1)}\oplus \dots \oplus T_{\alpha(t_k)}$.
The quadratic form $J''$ on $T_\alpha\Omega(t_0, \dots,
t_k)$ restricted to the time interval $[0, \tau]$ varies
continuously in $\tau$ even when $\tau$ equals $T_i$ for some $i$,
since the boundary terms in the second variation $J''(W,W)$ given by
\eqref{eq:twovar-summary} vanish when $W$ is a Jacobi field.  On any
subspace of $T_\alpha\Omega(t_0, \dots,
t_k)$ on which the index form is negative definite, it remains
negative definite on that space under small variations of $\tau$;
since $\lambda(\tau)$ is nondecreasing, then, we have
$\lambda(\tau-\ep) =\lambda(\tau)$ whenever $\ep>0$ is sufficiently
small.

We claim further that if $\alpha(\tau)$ is conjugate to $\alpha(0)$
with multiplicity $\nu$ then for $\ep>0$ sufficiently small,
$$
\lambda(\tau+\ep)=\lambda(\tau)+\nu;
$$
this will suffice to establish the theorem.
The fact that $\lambda(\tau+\ep)\leq \lambda(\tau)+\nu$ proceeds just
as in \cite{Mi:63}, as it depends just on the continuity of the index
form. It thus suffices to establish the reverse inequality.

If $\alpha(\tau)\notin Y$, then we also obtain
$\lambda(\tau+\ep)\geq \lambda(\tau)+\nu$ as in Milnor; we thus make a
few remarks on the case $\alpha(\tau) \in Y$.  In this case, let $W_1,\dots,
W_{\lambda(\tau)}$ denote the \reflective 
piecewise Jacobi fields vanishing at
the endpoints $t=0$ and $t=\tau$ (i.e., elements of $T_\alpha\Omega(t_0, \dots,
t_k)$) on which $J''$ is negative definite.  Let $Q_1,\dots, Q_\nu$
be independent reflected Jacobi fields vanishing at the endpoints $t=0$ and $t=\tau$.
Pick $X_k$ to be variation vector fields along $\alpha$ between times $0$ and
$\tau+\ep$ vanishing at the endpoints $t=0$ and $t=\tau+\ep$ so that
$$
\ang{D_t Q_i(\tau-), X_j(\tau-)}_g=\delta_{ij}.
$$
Extend the $Q_i$ and $W_j$ by zero on the interval $[\tau, \tau+\ep]$
(i.e.\ subsequent to reflection).  We note that just as in the
interior case, considering the second variation on $[0, \tau+\ep]$ we obtain
$$
J''(Q_i, W_j)=0
$$
and
$$
J''(Q_i, X_j)=2 \delta_{ij}.
$$
Here we crucially use the fact that the second variation form
\eqref{eq:twovar-summary} in the
case where 
$\alpha(\tau)\in Y$ with $Q_i(\alpha(\tau))=0$,
$Q_i$ a \reflective Jacobi field, and $Q_i=0$ for $t>\tau$ yields
$$
J''(Q_i, X_j) = \ang{D_t Q_i(\tau-), X_j(\tau-)}_g.
$$
The rest of the proof proceeds as usual: for $c$ sufficiently small,
the quadratic form $J''$ is negative definite on the span of
$$
W_1, \dots, W_{\lambda(\tau)}, c^{-1} Q_1-c X_1,\dots c^{-1} Q_\nu-c X_\nu.
$$

\end{proof}

\subsection{The Morse Index Theorem for composition of paths}

We also consider the Morse Index Theorem for the composition of paths: that is, given two \reflective \physs $\alpha_0:[0,T_0]\to M$ and $\alpha_1:[0,T_1]\to M$, with $\alpha_0(0)=p_0$, $\alpha_0(T_0) = p_1 = \alpha_1(0)$, and $\alpha_1(T_1) = p_2$, we consider the composite path $\alpha:[0,T]\to M$, where $T=T_0+T_1$, given by
\[\alpha(t) = \begin{cases} \alpha_0(t) & 0 \le t\le T_0 \\ \alpha_1(t-T_0) & T_0\le t\le T\end{cases}.\]
Note that if $\alpha_0$ and $\alpha_1$ are \reflective \physs, and
$p_1\not\in Y$, then $\alpha$ is a \reflective \phys iff
\begin{equation}\label{smooth}\dot\alpha_0(T_0) =
  \dot\alpha_1(0).\end{equation} \emph{The relation \eqref{smooth} will be
standing assumption for the remainder of this section.}
(Alternatively, we could have started with a
\reflective\phys $\alpha$, with $T_0\in(0,T)$ and $T_1 = T-T_0$, such
that $\alpha(T_0)\not\in Y$, and considered
$\alpha_0 = \alpha|_{[0,T_0]}$ and
$\alpha_1 = {\alpha(\cdot+T_0)|_{[0,T_1]}}$).

We can then consider the tangent space to $\alpha_i$ given by
\[T_{\alpha_i}\Omega_0(M;p_i,p_{i+1};T_i);\]
here, we explicitly incorporate the time $T_i$ into the notation. Note that $T_{\alpha_0}\Omega_0(M;p_0,p_1;T_0)$ can be identified with a subspace of $T_{\alpha}\Omega_0(M;p_0,p_2;T)$, the tangent space of the composite path, by identifying $Z\in T_{\alpha_0}\Omega_0(M;p_0,p_1;T_0)$ with the vector field on $\alpha(t)$ ($t\in[0,T]$) equaling $Z$ for $0\le t\le T_0$ and $0$ otherwise; such a vector field is indeed in the tangent space since $Z(T_0)=0$. We can similarly identify $T_{\alpha_1}\Omega_0(M;p_1,p_2;T_1)\subset T_\alpha\Omega_0(M;p_0,p_2;T)$, after shifting $t$ by $T_0$. Note that $J''$ is invariant under these identifications.

For $(x,y)$ sufficiently close to $(p_0,p_1)$, there exists a unique
\reflective\phys $\alpha_{0;x,y}$ (with time $T_0$) close to
$\alpha_0$ by Proposition \ref{prop:locdiffeo}. Define the action
$S_0(x,y)$ by  
\[S_0(x,y) = J[\alpha_{0;x,y}].\] 
Define $S_1(x,y)$ similarly, with respect to $\alpha_1$. 

Recalling the notation of \eqref{piecewiseJacobi},
let
  $T_\alpha\Omega(0,T_0,T)\subset T_\alpha\Omega_0(M;p_0,p_2;T)$
consist of continuous vector fields $W$ vanishing at the endpoints
$t=0,T$ such that $W|_{[0,T_0]}$ and $W(\cdot + T_0)|_{[0,T_1]}$ are
reflected Jacobi fields, allowing for a possible derivative
discontinuity at $t=T_0$. With this space of Jacobi fields defined, we have the following decomposition of $T_\alpha\Omega_0(M;p_0,p_2;T)$:

\begin{lemma}
Suppose $\alpha_i$ does not have a conjugate point at $t=T_i$ for $i=0,1$. Then
\[T_\alpha\Omega_0(M;p_0,p_2;T) = T_{\alpha_0}\Omega_0(M;p_0,p_1;T_0)
  \oplus T_{\alpha_1}\Omega_0(M;p_1,p_2;T_1) \oplus
  T_\alpha\Omega(0,T_0,T),\] with the direct sum orthogonal with
respect to $J''$.
\end{lemma}
\begin{proof}
Note for
$Z\in T_\alpha\Omega_0(M;p_0,p_2;T)$ that there is a unique reflected
Jacobi field $W_0$ on $[0,T_0]$ along $\alpha_0$ with $W_0(0)=0$ and $W_0(T_0) = Z(T_0)$ by Proposition \ref{prop:twopoint},
since we assume no conjugate points on $\alpha_0$ at $t=T_0$. Similarly,
there is a unique Jacobi field $W_1$ on $[0,T_1]$ along $\alpha_1$
with $W_1(0) = Z(T_0)$ and $W_1(T_1) = 0$. Now set
\[W(t) = \begin{cases} W_0(t) & 0\le t\le T_0 \\ W_1(t-T_0)& T_0 \le t\le T\end{cases},\]
$Z_0(t) = Z(t) - W_0(t)$ for $0\le t\le T_0$, and $Z_1(t) =
Z(t)-W_1(t-T_0)$ for $T_0\le t\le T$ (extending both $Z_0$ and $Z_1$
by zero to $[0,T]$). Then $Z_i\in
T_{\alpha_i}\Omega_0(M;p_i,p_{i+1};T_i)$ for $i=0,1$, $W\in
T_\alpha\Omega(0,T_0,T)$, and $Z = Z_0+Z_1+W$. Moreover, $Z_0$ and
$Z_1$ are orthogonal due to being zero on $[T_0,T]$ and $[0,T_0]$,
respectively, while $W$ is orthogonal to both $Z_0$ and $Z_1$ by 
Theorem~\ref{theorem:second}.
\end{proof}
Using Lemma \ref{lem:ind-add}, the following addition formula of indices is a direct corollary:
\begin{corollary}
\label{cor:ind-concat}
We have
\begin{align*}
\ind(J''|_{T_\alpha\Omega_0(M;p_0,p_2;T)}) &= \ind(J''|_{T_{\alpha_0}\Omega_0(M;p_0,p_1;T_0)}) + \ind(J''|_{T_{\alpha_1}\Omega_0(M;p_1,p_2;T_1)}) \\
&+\ind(J''|_{T_\alpha\Omega(0,T_0,T)}).
\end{align*}
\end{corollary}

We can in fact identify the last term in the above sum, using the action functions $S_0$ and $S_1$:

\begin{lemma}
\label{lem:hess-s0s1}
We have
\begin{equation}
\label{eq:hess-comp}
\ind(J''|_{T_\alpha\Omega(0,T_0,T)}) = \ind(\Hess|_{y=p_1}[S_0(p_0,y) + S_1(y,p_2)]).
\end{equation}
\end{lemma}
\begin{proof}
Note that
\[S_0(p_0,y)+S_1(y,p_2) = J[\alpha_{p_0,y,p_2}],\]
where
$\alpha_{p_0,y,p_2}:[0,T]\to M$ is a path such that
$\alpha_{p_0,y,p_2}|_{[0,T_0]}$ is a \reflective \phys from $p_0$ to
$y$ near $\alpha_0$, and $\alpha_{p_0,y,p_2}(\cdot+T_0)|_{[0,T_1]}$ is a \reflective
\phys from $y$ to $p_2$ near
 $\alpha_1$.  In particular,
for any $w\in T_{p_1}M$,
let 
\[y=y(\ep)=p_1+\ep w\]
in Riemannian normal coordinates at $p_1$.
Then
$\partial_{\ep}|_{\ep=0}[\alpha_{p_0,y,p_2}] = W$ where $W|_{[0,T_0]}$
and $W(\cdot + T_0)|_{[0,T_1]}$ are Jacobi fields, with $W(T_0) =
w$. Hence, 
\[\partial_{\ep\ep}^2|_{\ep=0}[S_0(p_0,y(\ep)) + S_1(y(\ep),p_2)] = \partial_{\ep\ep}^2|_{\ep=0}J[\alpha_{p_0,y,p_2}] = J''(W,W).\]
The left-hand side also equals
$\Hess|_{y=p_1}[S_0(p_0,y)+S_1(y,p_2)](w,w)$. Thus, identifying $T_\alpha\Omega(0,T_0,T)\cong T_{p_1}M$, via $W\mapsto W(T_0)$, with
\[J''(W,W) = \Hess|_{y=p_1}[S_0(p_0,y)+S_1(y,p_2)](W(T_0),W(T_0)),\]
thus giving \eqref{eq:hess-comp}.
\end{proof}

Putting together Corollary \ref{cor:ind-concat} and Lemma
\ref{lem:hess-s0s1} yields the following relation among the indices on $\alpha_0$, $\alpha_1$, and $\alpha$:
\begin{theorem}[Morse Index Theorem for composition of paths]\label{thm:addition}
Suppose $\alpha_i$ does not have a conjugate point at $t=T_i$ for $i=0,1$. Then
\begin{align*}
\ind(J''|_{T_\alpha\Omega_0(M;p_0,p_2;T)}) &= \ind(J''|_{T_{\alpha_0}\Omega_0(M;p_0,p_1;T_0)}) + \ind(J''|_{T_{\alpha_1}\Omega_0(M;p_1,p_2;T_1)}) \\
&+\ind(\Hess|_{y=p_1}[S_0(p_0,y) + S_1(y,p_2)]).
\end{align*}
\end{theorem}

\subsection{Periodic paths}
In order to discuss the Morse index theorem for periodic trajectories, we will need to discuss variations and Jacobi fields for periodic trajectories. We do so in this section before proceeding to the main theorem in the final section.

As we can freely choose a starting point for a periodic reflected physical path, we can without loss of generality let $\alpha(t)$ denote a periodic reflected physical path with
$\alpha(0)=\alpha(T) \notin Y$, and let
$\alpha(t,\ep,\delta)$ be a two-parameter family of periodic reflected
paths with $\alpha(t,0,0)=\alpha(t)$.  Denoting $\Rt$ and $\Kt$ the
reflection and kink times of $\alpha$, it will be convenient to
consider $\Kt_0:=\Kt\cup\{0\}$, i.e.\ to consider $t=0$ as an
additional kink time, owing to the possibility of $\alpha$ or its
variations being $\mathcal{C}^0$ but not $\Ct$ at $t=0$ (equivalently
$t=T$) when viewed as a periodic path. Correspondingly, for a vector
field $Z$ along $\alpha$, we write
\[\jump Z(0) := Z(0)-Z(T),\quad \ol{Z}(0) := \frac{1}{2}(Z(0)+Z(T)).\]
As we do not have vanishing of the endpoints at $t=0,T$ for
variational vector field $Z(t, \epsilon, \delta)\in T_\alpha\Omega_{\per}(M)$, in
contrast to equation \eqref{var1}, we obtain the first
variation formula
\begin{equation}
\label{eq:var1_closed}
    \begin{aligned}
    \frac{d}{d\epsilon}(J[\alpha(\cdot,\epsilon,\delta)]) &= -\int_0^T\bigang{\cov\dot \alpha+\nabla V(\alpha), Z(t,\ep,\delta)}\,dt\\
		&+ \sum_{i=0}^m\bigang{-\jump\dot\alpha(T_i(\ep,\delta),\ep,\delta),\ol{Z}(T_i(\ep,\delta),\ep,\delta)},
\end{aligned}
\end{equation}
where we use the convention that $T_0=0$ is independent of $\ep,\delta$.
The proof is similar to the proof for fixed endpoints, noting in this case that integration by parts \emph{does} produce boundary terms at $t=0$ and $t=T$; those boundary terms can be manipulated into the desired form in the same way as the terms at the other $T_i$. 

The second variation $J'':=\frac{\partial^2J}{\partial\epsilon \partial\delta}\big\rvert_{\substack{\epsilon=0\\ \delta=0}} $, in contrast to equation \eqref{eq:twovar-summary}, is given by
\begin{equation}
  \label{eq:var2_closed}
      \begin{aligned}
        J'' & (W,Z) =  -\int_0^T\bigang{D^2_t W + R(\dot\alpha, W) \dot \alpha+
            (\nabla^2 V) W, Z} \,dt 
             + \sum_{T_i\in\Rt} \bigg( -  \bigang{\jump D_t W, \overline{Z} } +2\ang{\ol{D_t W}, Z^-_\nrml } \\
  & +\bigang{2W_\nrml^-, (Z_\nrml^-/\dot \alpha_\nrml^-) \nabla V}
    + 2 \dot \alpha_1 \ang{\shape(\pa_\delta c), \ol{Z}}
  - 2\second(\pa_\delta c, \dot \alpha_\tgt)Z_1^-\bigg)_{T_i}
   - \sum_{T_i\in\Kt_0} \bigang{\jump D_tW (T_i),\ol{Z}(T_i) }.
      \end{aligned}
  \end{equation} 
\begin{remark}
    $J''(W,Z)$ is a quadratic form on the space $ T_\alpha\Omega_{\per}(M)$.
\end{remark}

Note that $\alpha(t)$ is a periodic reflected physical path (cf.~Definition \ref{def:path.global}). Let
\[\JJ(\alpha) = \{W\in T_\alpha\Omega(M)\,:\,W\text{ is a reflected Jacobi field along }\alpha|_{(0,T)}\},\]
i.e.\ the space of reflected Jacobi fields with no boundary conditions
at $0$, $T$. Let
\begin{align*}
\JJ_{C^0}(\alpha) &= \{W\in T_\alpha\Omega_{\per}(M)\,:\,W\text{ is a reflected Jacobi field along }\alpha|_{(0,T)}\}\\
&=\{W\in \JJ(\alpha)\,:\,W(0)=W(T)\}.
\end{align*}
This is a finite-dimensional vector space.
(In general the vector fields in $\JJ_{C^0}$ are continuous but not $C^1$ at the
endpoint $\alpha(0)=\alpha(T)$, hence the $C^0$ notation.) 
Note that if $p$ is not conjugate to itself along $\alpha$, i.e.\ there does not exist a nonzero $W\in \JJ(\alpha)$ with $W(0)=0=W(T)$, then
\begin{equation}
\label{eq:isom}
    \JJ(\alpha)\cong T_p(M)\times T_p(M): \  W\mapsto(W(0),W(T))
\end{equation}
is an isomorphism (since both spaces have dimension $2n$), and hence for any $w\in T_p(M)$, there exists a unique $W\in \JJ_{C^0}(\alpha)$ with $W(0)=w=W(T)$. This then gives the following:
\begin{lemma} 
Suppose $p$ is not conjugate to itself along $\alpha$. Then
\[T_\alpha\Omega_{\per}(M) = \JJ_{C^0}(\alpha)\oplus T_\alpha\Omega_0(M;p,p),\]
where the direct sum is orthogonal with respect to $J''$.
\end{lemma}
\begin{proof}
    Given any $V\in T_\alpha\Omega_{\per}(M)$, by the above isomorphism, its value $V_p$ uniquely determines a closed \reflective Jacobi field $W \in \JJ_{C^0}(\alpha)$. There exists a unique variational vector field $Z\in T_\alpha\Omega_0(M;p,p)$ such that $V=Z+W$. The orthogonality condition $J''(W,Z)=0$ is true as $W \in \JJ_{C^0}(\alpha)$ kills all terms in equation \eqref{eq:var2_closed} except the term $\bigang{\jump D_tW (0),\ol{Z}(0) }$, which is killed by $Z\in T_\alpha\Omega_0(M;p,p)$.
\end{proof}

Consequently we obtain:
\begin{corollary}
\label{cor:ind_idd}
The index of $J''$ on periodic paths is given by
\[\ind(J''|_{T_\alpha\Omega_{\per}(M)}) = \ind(J''|_{\JJ_{C^0}(\alpha)}) + \ind(J''|_{T_\alpha\Omega_0(M;p,p)}).\]
\end{corollary}

\subsection{The Morse index Theorem for periodic paths}
Note that Corollary \ref{cor:ind_idd} express the index of $J''$ on
periodic paths as the sum of the index $J''$ on fixed paths
$T_\alpha\Omega_0(M;p,p)$, which by Theorem \ref{thm:morse1} equals
the number of conjugate points along the path, and the index of $J''$
on $\JJ_{C^0}(\alpha)$. It suffices to identify this latter quantity. 

Fix a periodic \reflective \phys $\alpha$ with period
  $T$, and fix $p =\alpha(0)$.
  Assume that $p$ is not conjugate to itself along $\alpha$.  Then
  given any $x,y$ in a small neighborhood of $p$, recall from
  Proposition~\ref{prop:locdiffeo} that there exists a
  unique \reflective \phys $\alpha_{x,y}(t)$ close to
  $\alpha$ with 
  $$
  \alpha_{x,y}(0)=x,\ \alpha_{x,y}(T)=y.
    $$
We then as usual define the \emph{action}
    $$
S(x,y) := J[\alpha_{x,y}].
$$

\begin{lemma}
\label{lem:morse2}
\[\ind(J''|_{\JJ_{C^0}(\alpha)}) = \ind(\Hess|_{x=p}[S(x,x)])\]
\end{lemma}
\begin{proof}
  Let $\alpha_{x,y}(t)$ be defined as above; employing Riemann normal
  coordinates near $p$ to make sense of the following expressions, consider
  $$
\beta(t,\ep) := \alpha_{p+\ep w, p+\ep w}(t),
$$
a physical reflected path from $p+\ep w$ to itself.  Then
$$
W := \pa_\ep \beta(t,\ep)\rvert_{\ep=0}
$$
is a reflected Jacobi field in $\JJ_{C^0(\alpha)}$ by Proposition~\ref{prop:physicalvar},
with $W(0)=W(T)=w$.  Hence
$$
\Hess|_{x=p}[S(x,x)] (w,w) = \pa_{\ep\ep}^2 \rvert_{\ep=0}  J[\beta(t,\ep)] = J''(W,W),
$$
and the indices of these forms thus coincide (since the map $W \mapsto
W(0)$ is an isomorphism of $\JJ_{C^0}(\alpha)$ and $T_pM$).
\end{proof}

Combining Corollary \ref{cor:ind_idd}, Theorem \ref{thm:morse1} and Lemma \ref{lem:morse2}, we obtain
\begin{theorem}[Morse Index Theorem for periodic paths]
\label{thm:morse}
\begin{equation}\label{morseperiodic}\begin{aligned}\ind(J''|_{T_\alpha\Omega_{\per}(M)}) &= \ind (\Hess|_{x=p}[S(x,x)])
 \\ &\quad + \text{ number of conjugate points along }\alpha \text{ with
    respect to } p.\end{aligned}\end{equation}
\end{theorem}
Note that while the sum is an invariant of $\alpha$, each of the two
terms on the right-hand-side of \eqref{morseperiodic} may individually depend on
the choice of $p$ along $\alpha$ \cite[Section IV.B]{CrRoLi:90}.

\bibliographystyle{plain}
\bibliography{variational}

\end{document}